\documentclass[12pt]{article} 
\usepackage[utf8]{inputenc}
\usepackage[margin=1.06in]{geometry}
\usepackage{amsmath,amsthm}
\usepackage{mathtools}
\usepackage{amssymb}
\usepackage{amsopn}
\usepackage{bbm}
\usepackage{color}
\usepackage{hyperref}
\usepackage{url}
\usepackage{tikz}
\usepackage{tikz-cd}
\usepackage[font=footnotesize,labelfont=bf]{caption}
\usepackage{enumerate}
\usepackage{placeins}
\usepackage{eso-pic}
\usepackage{caption}

\usepackage{mathrsfs, stmaryrd}
\usepackage{a4wide}

\usepackage{hyperref}

\newtheorem{theorem}{Theorem}[section] 
\newtheorem{corollary}[theorem]{Corollary}
\newtheorem{definition}[theorem]{Definition}
\newtheorem{lemma}[theorem]{Lemma}
\newtheorem{remark}[theorem]{Remark}
\newtheorem{proposition}[theorem]{Proposition}

\def\ba#1\ea{\begin{align*}#1\end{align*}}
\def\beq#1\eeq{\begin{eqnarray}#1\end{eqnarray}}

\def \It\^o {It\^{o} }

\newcommand{\dexp}{\rm dexp}
\newcommand{\ad}{\rm ad}
\newcommand{\Id}{\rm Id}
\newcommand{\Tr}{\rm Tr}
\newcommand{\tr}{\rm tr}
\newcommand{\Cov}{\mathrm{Cov}}

\newcommand{\E}{\mathbb{E}}

\newcommand{\Sym}{{\rm Sym}}

\hypersetup{
  colorlinks = true, 
  urlcolor = blue, 
  linkcolor = blue, 
  citecolor = red 
}
\title{Higher-order variation and\\
pathwise It\^o calculus on manifolds}
\author{Rama Cont \\Mathematical Institute\\ University of Oxford }

\date{2026}

\begin{document}

\maketitle

\begin{abstract}
We develop an intrinsic
calculus for smooth functions of paths of arbitrarily low regularity on smooth manifolds.
The regularity of paths is defined in terms of a $p$-th variation
tensor along a sequence of partitions, for an arbitrary integer $p$;
this tensor is constructed as a local symmetric tensor measure along
the path. We define pathwise integrals of closed one-forms along paths
with finite $p$-th variation and derive a change of variable formula
for smooth functions of such paths. For $p=2$,
our results give a manifold version of H. F\"ollmer's pathwise It\^o calculus.

Our construction only requires an affine connection on the manifold and may be viewed as a higher-order analogue of
L. Schwartz's second-order differential geometry.
The connection provides a splitting of higher-order tangent vectors
into symmetric tensor components and leads to a geometric transfer
principle: the change-of-variable formula defines an intrinsic,
connection-independent functional of the reduced $p$-jet of the
test function, whose canonical highest-order component is determined
by the $p$-th variation tensor.

\noindent
Although our results are purely geometric, they apply to manifold-valued stochastic processes with highly irregular paths and yield a higher-order It\^o-type calculus for such processes. We illustrate this calculus for  exponential lifts of fractional Brownian motions to Riemannian  manifolds and Lie groups.
\end{abstract} 
{\it Mathematics Subject Classification}:	60H10  58A20 58C20 58C25 58C35 	  	49Q15 	60L99
\newpage

\tableofcontents

\newpage
\section{A calculus for irregular paths on smooth manifolds}

Starting with It\^o's
early work on stochastic analysis on manifolds
\cite{ito1950,ito1962}, various formulations have been proposed for
an intrinsic calculus for irregular manifold-valued paths. Stochastic differential geometry
 has   focused on diffusion processes on manifolds \cite{elworthy2006,elworthy2007,ikedawatanabe} and manifold-valued semimartingales \cite{emery,meyer1985,schwartz1984}, the latter naturally leading to Schwartz's ``second-order geometry" \cite{schwartz1982}. More recently, the extension of rough path theory to manifolds \cite{armstrong2022,cass2012,cass2015} have allowed to consider more general, and more irregular processes on manifolds \cite{BaudouinCoutin} but at the cost of requiring additional ingredients, such as rough-path enhancements, beyond the geometry of the underlying path.

It\^o's original approach \cite{ito1962} was based on approximating
a Brownian path on a manifold by the geodesic polygon obtained by
joining successive points sampled along a partition, and defining
parallel displacement along the Brownian path as the limit of
parallel displacement along these polygonal approximations.

In the present work, we show that this approach   may be extended to  paths of arbitrary regularity on a manifold.
Using It\^o's geodesic polygonal approximations and the corresponding 'logarithmic' increments, we construct the $p-$th order variation tensor of a path as a local tensor measure \cite{brena2024} along the path, in duality with symmetric $p-$forms.
 We develop an  intrinsic calculus for smooth functionals of irregular paths with finite $p-$th variation for any integer $p$, and show how closed one-forms may be integrated along such  paths. 
This approach allows us to define a higher-order It\^o-type calculus  \cite{CP2019} on manifolds with an affine connection.

For $p=2$, our results extend  F\"ollmer's pathwise It\^o calculus \cite{follmer1981} to manifold-valued paths. For general $p\in \mathbb{N},$ our construction lifts the  framework of Cont \& Perkowski \cite{CP2019} to the case of manifold-valued paths.

Our construction may also be viewed as a higher-order analogue of
Laurent Schwartz's second-order differential geometry
\cite{emery,emery2006,schwartz1982,schwartz1984}. Unlike the stochastic setting in which this geometry
originally arose, the present framework is purely pathwise. A choice
of affine connection provides a splitting of higher-order tangent
vectors into symmetric tensor components. In Section~\ref{sec:transfer-principle} we use this
splitting to formulate a geometric transfer principle: the
change-of-variable formula defines an intrinsic,
connection-independent functional of the reduced $p$-jet of the test
function, whose canonical highest-order component is determined by
the $p$-th variation tensor. 
The underlying principle is that jets are the intrinsic polynomial objects, while their decomposition into covariant homogeneous tensor components depends on the chosen geometric structure. This principle also underlies  the construction of regularity structures on manifolds by Hairer and Singh \cite{hairersingh}.

The higher-order geometric structures which arise naturally in this
extension have classical antecedents. Pohl \cite{Pohl1962}, building
on earlier work of Ehresmann \cite{ehresmann1,ehresmann2}, introduced higher-order
tangent, or osculating, bundles and showed that they form a natural
filtration whose successive quotients are symmetric tensor powers of
the tangent bundle. For $p=2$, this geometry is closely related to
the second-order differential geometry of L. Schwartz
\cite{schwartz1982,emery}.
\clearpage

Our construction leads to a higher order  It\^o-type calculus.
The  integral defined in Theorem \ref{thm:change-of-variable} is intrinsically of
It\^o type: it is defined as a non-anticipative limit of covariant left
Riemann--Taylor sums, with the $p$-th variation tensor appearing
explicitly as the highest-order correction term.
In this sense, our approach is a direct extension of Schwartz's It\^o calculus on manifolds \cite{schwartz1984} but without any probabilistic ingredient.
In the case $p=2$ our change of variable formula recovers the It\^o formula for Brownian motion on a manifold \cite{elworthy2007} and manifold-valued  semimartingales \cite{emery,schwartz1984}. 

However, our results are purely geometric, do not rely on any probabilistic construction and extend to arbitrary $p\in \mathbb{N}$, covering processes which are not semimartingales. 
We thus extend the stochastic change of variable formula on manifolds \cite{belopolskaya1990,elworthyICM,elworthy2007,ito1950}  to paths and processes of arbitrary roughness. We illustrate this in Section \ref{sec.examples} with the example of manifold-valued fractional processes and fractional processes in Lie groups.
The example of fractional rotations in
$\mathrm{SO}(3)$ in Section \ref{sec:SO3} illustrates how the $p$-th variation tensor
records the nonlinear geometry of a non-commutative Lie group through
the differential of the exponential map, while the resulting calculus
remains insensitive to torsion.
\vskip 1cm

Related integration theories have been developed using rough paths on manifolds \cite{armstrong2022,cass2012,cass2015}.
With the exception of \cite{armstrong2022},
these constructions are based on
geometric rough paths and lead to a
Stratonovich-type calculus. 
Non-geometric rough-path integration on manifolds was developed by Armstrong et al. \cite{armstrong2022} for paths with finite $q-$variation with $q<3$ and extended by Ferrucci \cite{ferrucci2025} to branched rough paths of arbitrary roughness.
These constructions require 
higher-order rough-path and bracket enhancements beyond the trace of the path.
In contrast, our construction  only involves the manifold geometry and function(al)s of the underlying path: it does not require any rough path enrichment of the paths considered, and all higher order tensors involved are constructed from the path itself by taking limits of geodesic polygonal approximations along a sequence of partitions, as in It\^o's original approach \cite{ito1962}. 

\paragraph{Outline} 

Section~\ref{sec.pthvariation} introduces the notion of finite $p$-th order variation tensor along a prescribed sequence of partitions for manifold-valued paths and derives its transformation rule under smooth maps. Section~\ref{sec:covariant-taylor} develops the covariant Taylor expansion along geodesics associated with an affine connection. Section~\ref{sec:change-of-variable} constructs the integral of a covariant $(p-1)$-jet, derives the corresponding
change of variable formula (Theorem \ref{thm:change-of-variable}),
extends the construction from exact to closed one-forms, and proves
uniform convergence and transformation identities for
$p$-th variation.  

Section~\ref{sec:transfer-principle} gives a geometric interpretation of this formula in terms
of higher-order tangent vectors and jets. Using the 
higher-order tangent geometry of Pohl \cite{Pohl1962}, we use an affine connection to
split the filtered bundle of order-$p$ tangent vectors into symmetric
tensor components. This identifies the change-of-variable formula
with an intrinsic functional of the reduced $p$-jet of the test
function and yields a higher-order analogue of the
Schwartz--\'Emery transfer principle \cite{emery2006,schwartz1982}.

Finally, Section~\ref{sec.examples} illustrates the theory through Riemannian Brownian motion and the exponential lift of fractional Brownian motion to spaces of constant curvature and the Lie group $SO(3)$, showing how the higher-order correction terms encode the geometry through the differential of the exponential map.

\section{Manifold-valued paths with finite $p$-th order variation}\label{sec.pthvariation}

Let $M$ be a smooth $d$-dimensional manifold without boundary and 
$p\geq 2$ be an integer. We assume that $M$ is endowed with a smooth
affine connection $\nabla$.
For basic definitions concerning affine connections,
geodesics, covariant derivatives and jet bundles we refer to
Bourguignon \cite{bourguignon2022} and Kobayashi \& Nomizu \cite{kobayashi1996}.
We denote by
\[
    \exp\colon \mathcal D\subset TM\longrightarrow M
\]
the exponential map associated with $\nabla$, defined on an open
neighborhood $\mathcal D$ of the zero section of $TM$. For
$(x,v)\in\mathcal D$, the curve
\[
    \gamma_{x,v}(s):=\exp_x(sv),
    \qquad 0\leq s\leq 1,
\]
is the geodesic starting from $x$ with initial velocity $v$, i.e.
\[
    \gamma_{x,v}(0)=x,
    \qquad
    \dot\gamma_{x,v}(0)=v,
    \qquad
    \nabla_{\dot\gamma_{x,v}}\dot\gamma_{x,v}=0.
\]
We  use the notation $    \Sym^p(TM)$ and
$    \Sym^p(T^*M)$
for the bundles of symmetric contravariant and covariant $p$-tensors,
respectively.
 $X\colon[0,T]\to M$ denotes a continuous path
, and
\[
    \pi=(\pi_n)_{n\geq 1},
    \qquad
    \pi_n=\{0=t_0^n<t_1^n<\cdots<t_{N_n}^n=T\},
\]
is a sequence of partitions of $[0,T]$ such that
$
    |\pi_n|
    :=
    \max_{0\leq i<N_n}(t_{i+1}^n-t_i^n)
    \longrightarrow 0.$
    
    We fix an auxiliary Riemannian metric $g$ on $M$ and use the
corresponding norms $|\cdot|$ on tangent and tensor spaces. The metric
is used only to formulate uniform remainder estimates and need not be
compatible with the connection $\nabla$.
Its only role in this
section is to provide norms on the tensor bundles and  define total
variation. Since the image $X([0,T])$ is compact, all choices of smooth
Riemannian metric give equivalent norms on the part of the bundles
visited by $X$, and hence none of the notions introduced below depends
on this auxiliary choice.
\subsection{Paths, increments and local tensor measures}

Following It\^o \cite{ito1962}, we define 'increments' of a manifold-valued path using a map $e:\mathcal U\subset M\times M\mapsto TM$:
\begin{definition}[Admissible increment map]
Let $\mathcal U\subset M\times M$ be an open neighborhood of the
diagonal $
    \Delta_M:=\{(x,x):x\in M\}.$
A smooth map $ e\colon\mathcal U\longrightarrow TM$ is an \emph{admissible increment map} if
$e(x,y)\in T_xM$ for every $(x,y)\in\mathcal U$ and, for every $x\in M$,
\begin{equation}
    e(x,x)=0_x,
    \qquad
    D_2e(x,x)=\Id_{T_xM}.
    \label{eq:admissible-increment}
\end{equation}
Here $D_2e(x,x)$ denotes differentiation with respect to the second
variable, using the canonical identification of the vertical tangent
space of $TM$ at $0_x$ with $T_xM$.
\end{definition}

The affine connection provides a canonical example: the exponential map
admits a smooth local inverse and we define the 'logarithmic increment' as
\begin{equation}
    e_\nabla(x,y)
    :=
    \exp_x^{-1}(y)\in T_xM.
    \label{eq:log-increment}
\end{equation}
The map $e_\nabla$ satisfies  the admissibility condition \eqref{eq:admissible-increment}.
Since $X$ is (uniformly) continuous, for every
neighborhood $\mathcal U$ of the diagonal, there exists $n_0$ such
that
$(X_{t_i^n},X_{t_{i+1}^n})\in\mathcal U$
for all $n\geq n_0$ and all $i$. Thus any admissible increment map is
defined on all sufficiently fine increments of the path. Set
\begin{equation}
    \Delta_i^{n,e}X
    :=
    e(X_{t_i^n},X_{t_{i+1}^n})
    \in T_{X_{t_i^n}}M.
    \label{eq:e-increment}
\end{equation}
Consider the pullback bundle
\[
    E_X^{(p)}
    :=
    X^*\Sym^p(TM)
    \longrightarrow [0,T]
\]
whose fiber at $t$ is
$
    (E_X^{(p)})_t
    =
    \Sym^p(T_{X_t}M).$
The dual bundle is
\[
    (E_X^{(p)})^*
    =
    X^*\Sym^p(T^*M).
\]

Associated with $X$, $\pi_n$ and an admissible increment map $e$, we
define the atomic tensor measure
\begin{equation}
    \mu_{X,n}^{(p),e}
    :=
    \sum_{i=0}^{N_n-1}
    \bigl(\Delta_i^{n,e}X\bigr)^{\otimes p}
    \,\delta_{t_i^n}.
    \label{eq:partition-tensor-measure}
\end{equation}
Thus $\mu_{X,n}^{(p),e}$ is a finite
$E_X^{(p)}$-valued Radon measure on $[0,T]$.
More explicitly, if
$ A\in  C\bigl([0,T];X^*\Sym^p(T^*M)\bigr)$
is a continuous symmetric covariant $p$-tensor field along $X$, then
\begin{equation}
    \int_0^T
    \big\langle A_t,d\mu_{X,n}^{(p),e}(t)\big\rangle
    =
    \sum_{i=0}^{N_n-1}
    A_{t_i^n}
    \left(
        \bigl(\Delta_i^{n,e}X\bigr)^{\otimes p}
    \right).
    \label{eq:pairing-tensor-measure}
\end{equation}

For a Borel set $B\subset[0,T]$, we may equivalently define
\[
    \mu_{X,n}^{(p),e}(B)(A)
    :=
    \sum_{\substack{0\leq i<N_n\\ t_i^n\in B}}
    A_{t_i^n}
    \left(
        \bigl(\Delta_i^{n,e}X\bigr)^{\otimes p}
    \right).
\]
In particular, if $A$ vanishes on $B$, then
$ \mu_{X,n}^{(p),e}(B)(A)=0.$
Hence $\mu_{X,n}^{(p),e}$ is a local vector measure, in the sense of
Brena--Gigli \cite{brena2024}, acting in duality with continuous symmetric covariant
$p$-tensor fields along $X$.
The total variation of \eqref{eq:partition-tensor-measure} is
\begin{equation}
    \bigl|\mu_{X,n}^{(p),e}\bigr|
    =
    \sum_{i=0}^{N_n-1}
    \bigl|\Delta_i^{n,e}X\bigr|_g^p\,
    \delta_{t_i^n},
    \label{eq:total-variation-partition-measure}
\end{equation}
where we use the tensor norm induced by $g$ and the identity $ |v^{\otimes p}|=|v|^p.$
\subsection{Paths with finite $p$-th variation along a sequence of partitions}

We first isolate the quantitative assumption which controls the
partition sums.

\begin{definition}[Finite $p$-energy along $\pi$]
Let $e$ be an admissible increment map. We say that $X$ has
\emph{finite $p$-energy along $\pi$ with respect to $e$} if there exists $n_0\in \mathbb{N}$ such that
\begin{equation}
    \sup_{n\geq n_0}
    \sum_{i=0}^{N_n-1}
    \bigl|
        \Delta_i^{n,e}X
    \bigr|_g^p
    <\infty.
    \label{eq:bounded-p-energy}
\end{equation}
\end{definition}

The next definition includes both boundedness of the $p$-energy and the
existence of a tensorial limit.

\begin{definition}[$p$-th variation tensor along $\pi$]
Let $e$ be an admissible increment map. We say that
$X\in C([0,T],M)$ has \emph{finite $p$-th order variation along $\pi$}
if \eqref{eq:bounded-p-energy} holds and there exists a finite
$X^*\Sym^p(TM)$-valued Radon measure
\[
    [X]_{\pi}^{(p)}
    \in
    \mathcal M\bigl(
        [0,T];X^*\Sym^p(TM)
    \bigr)
\]
such that
\begin{equation}
    \mu_{X,n}^{(p),e}
    \stackrel{*}{\rightharpoonup}
    [X]_{\pi}^{(p)}
    \qquad{\rm as }\quad n\to\infty.
    \label{eq:weak-tensor-convergence}
\end{equation}
That is,
\begin{equation}
  \forall  A\in
    C\bigl([0,T];X^*\Sym^p(T^*M)\bigr),
\qquad
    \lim_{n\to\infty}
    \sum_{i=0}^{N_n-1}
    A_{t_i^n}
    \left(
        \bigl(\Delta_i^{n,e}X\bigr)^{\otimes p}
    \right)
    =
    \int_0^T
    A_t\bigl(d[X]_{\pi,t}^{(p)}\bigr).
    \label{eq:def-p-variation}
\end{equation}\label{def:p-variation}
\end{definition}
We denote $V_p(M,\pi)$
the set of continuous paths
    $X:[0,T]\mapsto M$ 
which satisfy
Definition~\ref{def:p-variation}:
\[
    V_p(M,\pi)
    :=
    \left\{
        X\in C([0,T],M):
        X \text{ admits a $p$-th variation tensor }
        [X]_\pi^{(p)}
        \text{ along }\pi 
    \right\}.
\]
Here weak-* convergence is understood in duality with continuous
symmetric covariant $p$-tensor fields along $X$:
\[
    \mu_n\stackrel{*}{\rightharpoonup}\mu
\iff
  \forall A\in
    C\!\left([0,T];
    X^*\operatorname{Sym}^p(T^*M)\right),\quad   \int_{[0,T]}\langle A_t,d\mu_n(t)\rangle
    \longrightarrow
    \int_{[0,T]}\langle A_t,d\mu(t)\rangle
\]
which in the case of the discrete measures \eqref{eq:partition-tensor-measure} leads to \eqref{eq:def-p-variation}.

At this stage Definition \ref{def:p-variation} may appear to depend on the increment map $e$. We now show that it does not.
\begin{lemma}[Comparison of admissible increments]
\label{lem:comparison-increments}
Let $e$ and $\widetilde e$ be two admissible increment maps. Let
$K\subset M$ be compact. There exists a neighborhood $U$ of the diagonal in $K\times K$
and 
$C_K>0$ such that for $(x,y)\in U$
\begin{equation}
    \left|
        e(x,y)^{\otimes p}
        -
        \widetilde e(x,y)^{\otimes p}
    \right|
    \leq
    C_K\,d_g(x,y)^{p+1}
    \label{eq:tensor-comparison}
\end{equation}
and
\begin{equation}
    C_K^{-1}d_g(x,y)
    \leq
    |e(x,y)|_g
    \leq
    C_Kd_g(x,y),
    \label{eq:increment-distance-equivalence}
\end{equation}
and the same estimate holds for $\widetilde e$.
\end{lemma}

\begin{proof}
Set
\[
    r(x,y):=e(x,y)-\widetilde e(x,y)\in T_xM.
\]
Fix $x_0\in K$ and choose a coordinate neighborhood $U$ of $x_0$
together with a smooth trivialization
$TM|_U\simeq U\times\mathbb R^d.$
Since   $e$ and $\widetilde e$ are
admissible,
\[
    r(x,x)=0,
    \qquad
    D_2r(x,x)=0.
\]
Taylor's theorem in the second variable therefore yields
\[
    |r(x,y)|
    \leq
    C\,|\varphi(y)-\varphi(x)|^2
\]
for $x,y$ in a sufficiently small relatively compact coordinate
neighborhood.
Hence
\[
    |r(x,y)|_g
    \leq
    C\,d_g(x,y)^2.
\]
A finite covering of $K$ gives
\eqref{eq:tensor-comparison} for $p=2$ with a uniform constant $C_K$.
Similarly, the conditions
\[
    e(x,x)=0,
    \qquad
    D_2e(x,x)=\Id
\]
give 
$e(x,y)
    =
    \varphi(y)-\varphi(x)
    +
    O\bigl(|\varphi(y)-\varphi(x)|^2\bigr)$ in local coordinates.
Shrinking the neighborhood of the diagonal if necessary therefore
gives \eqref{eq:increment-distance-equivalence}. The argument for
$\widetilde e$ is identical.
Admissibility gives 
\[
    |e(x,y)|+|\widetilde e(x,y)|\leq C d_g(x,y),
    \qquad
    |e(x,y)-\widetilde e(x,y)|\leq C d_g(x,y)^2.
\]
uniformly on compact sets. Hence, using
\[
    u^{\otimes p}-v^{\otimes p}
    =
    \sum_{k=0}^{p-1}
    u^{\otimes k}\otimes(u-v)\otimes
    v^{\otimes(p-1-k)},
\]
we obtain
\[
    |e(x,y)^{\otimes p}-\widetilde e(x,y)^{\otimes p}|
    \leq C d_g(x,y)^{p+1}.
\]

\end{proof}

\begin{proposition}[Independence of the increment map]
\label{prop:independence-e}
Let $e$ and $\widetilde e$ be two admissible increment maps and let
$X\in C([0,T],M)$. Assume that
\[
    \sup_{n\geq n_0}
    \sum_{i=0}^{N_n-1}
    |\Delta_i^{n,e}X|_g^p
    <\infty.
\]
Then
\begin{equation}
   \mathop{\sup}_{n\geq n_0}
    \sum_{i=0}^{N_n-1}
    |\Delta_i^{n,\widetilde e}X|_g^p
    <\infty,\qquad \left\|
        \mu_{X,n}^{(p),e}
        -
        \mu_{X,n}^{(p),\widetilde e}
    \right\|_{\mathrm{TV}}
    \longrightarrow0.
    \label{eq:TV-independence-e}
\end{equation}
In particular, if
\[
    \mu_{X,n}^{(p),e}
    \stackrel{*}{\rightharpoonup}
    [X]_{\pi}^{(p)},
\quad{\rm then}\quad
    \mu_{X,n}^{(p),\widetilde e}
    \stackrel{*}{\rightharpoonup}
    [X]_{\pi}^{(p)}.
\]
Thus the finite $p$-energy property and the $p$-th variation tensor do not depend on the choice of admissible increment map.
\end{proposition}

\begin{proof}
    $K:=X([0,T])$
 is compact. Define $
    \omega_X(\pi_n)
    :=
    \max_{0\leq i<N_n}
    d_g(X_{t_i^n},X_{t_{i+1}^n}).$
Uniform continuity of $X$ and $|\pi_n|\to0$ imply
$\omega_X(\pi_n)\to 0.$
By Lemma~\ref{lem:comparison-increments}, for sufficiently large
$n$,
\[
    |\Delta_i^{n,\widetilde e}X|_g
    \leq
    C_K|\Delta_i^{n,e}X|_g,
\]
uniformly in $i$. Hence
\[
    \sum_i
    |\Delta_i^{n,\widetilde e}X|_g^p
    \leq
    C_K^p
    \sum_i
    |\Delta_i^{n,e}X|_g^p,
\]
which proves boundedness of the $p$-energy computed with $\widetilde e$. Lemma~\ref{lem:comparison-increments} gives
\[
\begin{split}
    \left\|
        \mu_{X,n}^{(p),e}
        -
        \mu_{X,n}^{(p),\widetilde e}
    \right\|_{\mathrm{TV}}
    &\leq
    C_K
    \sum_i
    d_g(X_{t_i^n},X_{t_{i+1}^n})^{p+1}
    \\
    &\leq
    C_K\,
    \omega_X(\pi_n)
    \sum_i
    d_g(X_{t_i^n},X_{t_{i+1}^n})^p.
\end{split}
\]
By \eqref{eq:increment-distance-equivalence}, the last sum is bounded
uniformly in $n$ by a constant multiple of
$
    \sum_i|\Delta_i^{n,e}X|_g^p.$
Therefore
\[
    \left\|
        \mu_{X,n}^{(p),e}
        -
        \mu_{X,n}^{(p),\widetilde e}
    \right\|_{\mathrm{TV}}
    \leq
    C\,\omega_X(\pi_n)
    \longrightarrow0.
\]
This proves \eqref{eq:TV-independence-e}. Weak-* convergence of one
sequence of measures is therefore equivalent to weak-* convergence of
the other, and the two limits coincide.
\end{proof}

\begin{remark}[Intrinsic definition]
Proposition~\ref{prop:independence-e} shows that the $p$-th variation
tensor depends only on the first-order behavior of the increment map
along the diagonal. Any two smooth maps satisfying
\[
    e(x,x)=0,
    \qquad
    D_2e(x,x)=\Id_{T_xM},
\]
lead to the same limiting tensor measure.
Consequently, when a connection $\nabla$ is fixed, one may use without loss of generality the
intrinsic logarithmic increment
\[
    e_\nabla(x,y)=\exp_x^{-1}(y).
\]

\end{remark}

\begin{remark}[Embedded manifolds]
Suppose $(M,g)$ is isometrically embedded in $\mathbb R^N$ and
let
\[
    P_x\colon\mathbb R^N\longrightarrow T_xM
\]
denote Euclidean orthogonal projection. Then
$   e(x,y):=P_x(y-x)$
is an admissible increment map.
For the Levi--Civita connection of the induced metric one has 
\[
    P_x(y-x)
    =
    \exp_x^{-1}(y)
    +
    O\bigl(d_g(x,y)^3\bigr).
\]
in a neighborhood of the diagonal.
In particular, Proposition~\ref{prop:independence-e} implies that the
projected ambient increment and the intrinsic geodesic increment yield
the same $p$-th variation tensor whenever either construction exists.
\end{remark}
\begin{remark}[$p-$th variation tensor as a local vector measure]
The $p-$th variation 
\[
    [X]_{\pi}^{(p)}
    \in
    \mathcal M\bigl(
        [0,T];X^*\Sym^p(TM)
    \bigr)
\]
is a symmetric $p$-tensor-valued measure along $X$. Equivalently, it
defines a local vector measure, in the sense of Brena-Gigli \cite{brena2024}, acting on sections of
$   X^*\Sym^p(T^*M).$
For a Borel set $B\subset[0,T]$ and
$    A\in
    C\bigl([0,T];X^*\Sym^p(T^*M)\bigr),$
we set
\[
    [X]_{\pi}^{(p)}(B)(A)
    :=
    \int_B A_t\bigl(d[X]_{\pi,t}^{(p)}\bigr).
\]
Its locality is expressed by
\[
    A|_B=0
    \quad\Longrightarrow\quad
    [X]_{\pi}^{(p)}(B)(A)=0.
\]
One may further push this measure forward by $X$ to obtain a local
vector measure on $M$ supported on the trace $X([0,T])$. The
time-parametrized tensor measure along $X$, however, is the more
natural object for the pathwise calculus developed below.
\end{remark}
\subsection{Transformation under smooth maps}
\label{subsec:smooth-map-transformation}

A basic requirement for an intrinsic notion of
higher-order variation on a manifold is naturality under smooth maps.
Transformation rules for quadratic variation under nonlinear maps were studied by
F\"ollmer \cite{follmer1981} in the vector case and in
\cite{ananova2017} under smooth functionals. Analogous results for $p-$th variation in the vector case were given in \cite{CP2019}.
We now establish the corresponding tensorial transformation rule for
the $p$-th variation of manifold-valued paths.

Let $M$ and $N$ be smooth finite-dimensional manifolds, let
$ \Phi:M\longrightarrow N$
be smooth, and let $    X:[0,T]\longrightarrow M$
be continuous. Set 
    $Y=\Phi\circ X.$
For every $t\in[0,T]$, the differential of $\Phi$ induces a linear
map
\[
    (D\Phi_{X_t})^{\otimes p}:
    \Sym^p(T_{X_t}M)
    \longrightarrow
    \Sym^p(T_{Y_t}N).
\]
Thus $D\Phi$ induces a continuous bundle morphism
\[
    (D\Phi_X)^{\otimes p}:
    X^*\Sym^p(TM)
    \longrightarrow
    Y^*\Sym^p(TN)
\]
over the identity map of $[0,T]$.

\begin{definition}[Fiberwise pushforward of a tensor measure]
\label{def:fiberwise-pushforward}
Let
$\mu\in
    \mathcal M\bigl([0,T];X^*\Sym^p(TM)\bigr)$
be a finite symmetric $p$-tensor-valued measure along $X$.
The \emph{fiberwise pushforward of $\mu$ by $\Phi$} is the measure
\[
    \Phi_*^{(p)}\mu
    \in
    \mathcal M\bigl([0,T];Y^*\Sym^p(TN)\bigr)
\]
defined by
\begin{equation}
    d\bigl(\Phi_*^{(p)}\mu\bigr)_t
    :=
    (D\Phi_{X_t})^{\otimes p}\,d\mu_t .
    \label{eq:fiberwise-pushforward-formal}
\end{equation}
It is characterized by
\begin{equation}
\begin{split}
 \forall  A\in
    C\bigl([0,T];Y^*\Sym^p(T^*N)\bigr),\quad   \int_0^T
    \left\langle
        A_t,
        d\bigl(\Phi_*^{(p)}\mu\bigr)_t
    \right\rangle
    &=
    \int_0^T
    \left\langle
        (D\Phi_{X_t})^{*\otimes p}A_t,
        d\mu_t
    \right\rangle.
\end{split}
\label{eq:fiberwise-pushforward-duality}
\end{equation}
Here
\[
    (D\Phi_{X_t})^{*\otimes p}:
    \Sym^p(T_{Y_t}^*N)
    \longrightarrow
    \Sym^p(T_{X_t}^*M)
\]
denotes the dual tensor map.
\end{definition}

\begin{remark}{\em 
The operation $\Phi_*^{(p)}$ is a pushforward in the tensor fibers
only: the base space $[0,T]$ is unchanged. It should therefore be
distinguished from the pushforward of a (local vector) measure on $[0,T]$ by a map
of the base space  in Brena and Gigli \cite{brena2024}.}
\end{remark}

\begin{proposition}[Transformation of the $p$-th variation tensor]
\label{prop:smooth-map-transformation}
Let $p\geq2$, 
    $X\in V_p(M,\pi)$
and  $\Phi:M\to N$ a smooth map. Then $
    \Phi\circ X\in V_p(N,\pi)$
and its $p$-th variation tensor is the fiberwise pushforward of that
of $X$:
\begin{equation}
    [\Phi\circ X]_\pi^{(p)}
    =
    \Phi_*^{(p)}[X]_\pi^{(p)}.
\label{eq:smooth-map-transformation}
\end{equation}
Equivalently,
\begin{equation}
\boxed{
    d[\Phi\circ X]_{\pi,t}^{(p)}
    =
    (D\Phi_{X_t})^{\otimes p}
    d[X]_{\pi,t}^{(p)} .
}
\label{eq:smooth-map-transformation-differential}
\end{equation}
\end{proposition}

\begin{proof}
Let $e_M$ and $e_N$ be admissible increment maps on $M$ and $N$.    $K:=X([0,T])$ is compact.
For $x,y\in K$ sufficiently close, define
\[
    R_\Phi(x,y)
    :=
    e_N(\Phi(x),\Phi(y))
    -
    D\Phi_x\,e_M(x,y)
    \in T_{\Phi(x)}N.
\]
By admissibility, $R_\Phi(x,x)=0,$
and
\[
\begin{split}
    D_2R_\Phi(x,x)
    &=
    D_2e_N(\Phi(x),\Phi(x))\circ D\Phi_x
    -
    D\Phi_x\circ D_2e_M(x,x)
    =D\Phi_x-D\Phi_x
    =0 .
\end{split}
\]
Taylor expansion in the second variable, uniformly for $x\in K$,
therefore gives
\begin{equation}
    |R_\Phi(x,y)|
    \leq C_K d_M(x,y)^2 .
    \label{eq:Phi-increment-error}
\end{equation}
Hence
\begin{equation}
    e_N(\Phi(x),\Phi(y))
    =
    D\Phi_x\,e_M(x,y)
    +
    O\bigl(d_M(x,y)^2\bigr),
    \label{eq:Phi-increment-first-order}
\end{equation}
uniformly on $K$.
Since $D\Phi$ is bounded on $K$, and admissible increments are
comparable with Riemannian distance on compact sets,
\[
    |e_N(\Phi(x),\Phi(y))|
    +
    |D\Phi_xe_M(x,y)|
    \leq C_K d_M(x,y).
\]
Using
\[
    u^{\otimes p}-v^{\otimes p}
    =
    \sum_{k=0}^{p-1}
    u^{\otimes k}\otimes
    (u-v)\otimes
    v^{\otimes(p-1-k)},
\]
we obtain from \eqref{eq:Phi-increment-error}
\begin{equation}
\left|
e_N(\Phi(x),\Phi(y))^{\otimes p}
-
(D\Phi_x)^{\otimes p}
e_M(x,y)^{\otimes p}
\right|
\leq
C_Kd_M(x,y)^{p+1}.
\label{eq:Phi-tensor-increment-error}
\end{equation}
Write
\[
    \Delta_i^{n,M}X
    :=
    e_M(X_{t_i^n},X_{t_{i+1}^n}),
\qquad
    \Delta_i^{n,N}Y
    :=
    e_N(Y_{t_i^n},Y_{t_{i+1}^n}).
\]
Equation \eqref{eq:Phi-increment-first-order} and boundedness of
$D\Phi$ on $K$ imply
\[
    |\Delta_i^{n,N}Y|
    \leq C_K|\Delta_i^{n,M}X|
\]
for all sufficiently large $n$. Since $X$ has finite $p$-energy
along $\pi$, so does $Y$.

Now consider
\[
    \mu_{Y,n}^{(p),e_N}
    =
    \sum_i
    (\Delta_i^{n,N}Y)^{\otimes p}\delta_{t_i^n}
\]
and
\[
    \Phi_*^{(p)}\mu_{X,n}^{(p),e_M}
    =
    \sum_i
    (D\Phi_{X_{t_i^n}})^{\otimes p}
    (\Delta_i^{n,M}X)^{\otimes p}
    \delta_{t_i^n}.
\]
Set
$   \omega_X(\pi_n)
    =
    \max_i
    d_M(X_{t_i^n},X_{t_{i+1}^n}).$
Then $\omega_X(\pi_n)\to0$, and
\eqref{eq:Phi-tensor-increment-error} yields
\begin{align*}
&
\left\|
    \mu_{Y,n}^{(p),e_N}
    -
    \Phi_*^{(p)}\mu_{X,n}^{(p),e_M}
\right\|_{\mathrm{TV}}
\leq
C_K
\sum_i
d_M(X_{t_i^n},X_{t_{i+1}^n})^{p+1}
\\
&\qquad\leq
C_K\omega_X(\pi_n)
\sum_i
d_M(X_{t_i^n},X_{t_{i+1}^n})^p
\longrightarrow0.
\end{align*}
The last sum is uniformly bounded by finite $p$-energy and
Lemma~\ref{lem:comparison-increments}.
Since
$
    \mu_{X,n}^{(p),e_M}
    \stackrel{*}{\rightharpoonup}
    [X]_\pi^{(p)},$
Definition~\ref{def:fiberwise-pushforward} implies
\[
    \Phi_*^{(p)}\mu_{X,n}^{(p),e_M}
    \stackrel{*}{\rightharpoonup}
    \Phi_*^{(p)}[X]_\pi^{(p)}.
\]
Together with the preceding total-variation estimate this gives
$    \mu_{Y,n}^{(p),e_N}
    \stackrel{*}{\rightharpoonup}
    \Phi_*^{(p)}[X]_\pi^{(p)}.$
Hence $Y=\Phi\circ X$ belongs to $V_p(N,\pi)$ and
\eqref{eq:smooth-map-transformation} follows.
\end{proof}

\begin{remark}
Although second derivatives of $\Phi$ enter the estimate
\eqref{eq:Phi-increment-error}, only the first derivative
$D\Phi$ appears in the limiting transformation rule
\eqref{eq:smooth-map-transformation-differential}. 
Pohl \cite{Pohl1962} associates with every smooth
map $\Phi:M\to N$ a $p$-th differential
\[
    \Phi_p:\tau^{(p)}M\to\tau^{(p)}N
\]
compatible with the natural filtration. On the highest-order quotient
$\tau^{(p)}M/\tau^{(p-1)}M\simeq\Sym^p(TM),$
the induced map is $(D\Phi)^{\otimes p}$. Proposition~\ref{prop:smooth-map-transformation}
shows that the $p$-th variation tensor transforms precisely according
to this highest-order component.
\end{remark}
\section{Covariant Taylor expansion along geodesics}
\label{sec:covariant-taylor}
 The symmetrized covariant derivative used below is developed in Palais \cite{palais}. Various forms of covariant Taylor expansions along geodesics are studied in \cite{avramidi2000,palais,rempala1988,widom1980}. We provide here a self-contained derivation under the exact assumptions used in the sequel.

Recall that $M$ is  endowed with a smooth
affine connection $\nabla$. The connection induces, in the usual way,
connections on the cotangent bundle and on all tensor bundles over
$M$. 

\begin{remark}[No assumption on torsion]\label{rem.notorsion}
No torsion-free assumption on $\nabla$ is required anywhere in the
construction. The geodesic equation, and hence the exponential map,
depends only on the symmetric part of the connection. Moreover, for
every $k\geq1$,
\[
    (\nabla^{(k)}f)_x(v^{\otimes k})
    =
    \left.
    \frac{d^k}{ds^k}
    f(\exp_x(sv))
    \right|_{s=0},
\]
so by polarization the symmetrized covariant derivatives
$\nabla^{(k)}f$ also depend only on the symmetric part of $\nabla$.
Finally, the jet map
\[
    j_x^p f
    \longmapsto
    \bigl(
        f(x),\nabla f(x),\nabla^{(2)}f(x),\ldots,
        \nabla^{(p)}f(x)
    \bigr)
\]
is an isomorphism for an arbitrary affine connection: in local
coordinates it is triangular with respect to the ordinary partial
derivatives, with identity on each highest-order component.
Thus torsion plays no role in the present calculus: replacing an
affine connection by its symmetrization leaves both the geodesic
increments and the symmetrized covariant derivatives unchanged.
This should be contrasted with rough-path constructions involving
parallel transport or Cartan development, where the full connection
is used and torsion may affect the resulting parallel transport \cite{armstrong2022}.
\end{remark}
 \paragraph{Covariant Taylor formula}
The following identity is the basis of the covariant Taylor formula.
Let $k\geq1$, let $f\in C^k(M)$ and let
$\gamma\colon[0,1]\longrightarrow M$ 
be a geodesic. Iteration of the covariant Leibniz
rule yields 
\begin{equation}
    \frac{d^k}{ds^k}f(\gamma(s))
    =
    (\nabla^k f)_{\gamma(s)}
    \bigl(\dot\gamma(s)^{\otimes k}\bigr)=
    (\nabla^{(k)}f)_{\gamma(s)}
    \bigl(\dot\gamma(s)^{\otimes k}\bigr).
    \label{eq:derivatives-along-geodesic}
\end{equation}
In particular if $
    \gamma_{x,v}(s)=\exp_x(sv)$
then, for $1\leq k\leq p$,
\begin{equation}
    \left.
    \frac{d^k}{ds^k}
    f(\exp_x(sv))
    \right|_{s=0}
    =
    (\nabla^k f)_x(v^{\otimes k})
    =
    (\nabla^{(k)}f)_x(v^{\otimes k}).
    \label{eq:derivatives-exp-zero}
\end{equation} 
We derive here an intrinsic Taylor formula 
whose remainder
is $o(|v|^p)$, uniformly when the base point varies over a compact set.

\begin{theorem}[Covariant Taylor formula with uniform remainder bound]
\label{thm:covariant-taylor-peano}
Let $p\geq1$ and let $f\in C^p(M)$. Let $K\subset M$ be compact.
Then there exist $r_K>0$ and a nondecreasing function
\[
\omega_{f,K}\colon[0,r_K]\longrightarrow[0,\infty )\qquad
{\rm such\  that}\quad
    \lim_{r\downarrow0}\omega_{f,K}(r)=0
\]
and, for every $x\in K$ and every $v\in T_xM$ with
$|v|\leq r_K$,
\begin{equation}
    f(\exp_xv)
    =
    \sum_{k=0}^p
    \frac1{k!}
    (\nabla^{(k)}f)_x(v^{\otimes k})
    +
    R_{p,f}(x,v),
    \label{eq:covariant-taylor-peano}
\end{equation}
where
\begin{equation}
    |R_{p,f}(x,v)|
    \leq
    \omega_{f,K}(|v|)\,|v|^p
    \label{eq:uniform-peano-bound}
\end{equation}
Here we use the convention $
    \nabla^{(0)}f=f.$
\end{theorem}

\begin{proof}
Since $K$ is compact and the domain $\mathcal D$ of the exponential
map is an open neighborhood of the zero section, there exists
$r_K>0$ such that
$    (x,v)\in\mathcal D $
whenever $x\in K$ and $|v|\leq r_K$, after decreasing $r_K$ if
necessary.
Define
\[
    h(s):=f(\gamma_{x,v}(s))
          =f(\exp_x(sv)).
\]
Since $f\in C^p(M)$,  $h\in C^p([0,1])$.
The one-dimensional Taylor formula  gives
\[
    h(1)
    =
    \sum_{k=0}^{p-1}
    \frac{h^{(k)}(0)}{k!}
    +
    \frac1{(p-1)!}
    \int_0^1
    (1-s)^{p-1}h^{(p)}(s)\,ds.
\]
By ~\eqref{eq:derivatives-along-geodesic},
\[
    h^{(k)}(0)
    =
    (\nabla^{(k)}f)_x(v^{\otimes k})\qquad{\rm and}\qquad
    h^{(p)}(s)
    =
    (\nabla^{(p)}f)_{\gamma_{x,v}(s)}
    \bigl(\dot\gamma_{x,v}(s)^{\otimes p}\bigr).
\]
Substitution yields:
\[
\begin{split}
    f(\exp_xv)
    &=
    \sum_{k=0}^{p-1}
    \frac1{k!}
    (\nabla^{(k)}f)_x(v^{\otimes k})
    +
    \frac1{(p-1)!}
    \int_0^1
    (1-s)^{p-1}
    (\nabla^{(p)}f)_{\gamma_{x,v}(s)}
    \bigl(\dot\gamma_{x,v}(s)^{\otimes p}\bigr)
    \,ds.
\end{split}
\]
Since
\[
    \int_0^1(1-s)^{p-1}\,ds=\frac1p,
\]
we may add and subtract
\[
    \frac1{p!}
    (\nabla^{(p)}f)_x(v^{\otimes p})
\]
to obtain
\[
    f(\exp_xv)
    =
    \sum_{k=0}^p
    \frac1{k!}
    (\nabla^{(k)}f)_x(v^{\otimes k})
    +
    R_{p,f}(x,v),
\]
where
\begin{equation}
\begin{split}
    R_{p,f}(x,v)
    &:=
    \frac1{(p-1)!}
    \int_0^1(1-s)^{p-1}
    \Big[
       (\nabla^{(p)}f)_{\gamma_{x,v}(s)}
       \bigl(\dot\gamma_{x,v}(s)^{\otimes p}\bigr)
       -
       (\nabla^{(p)}f)_x(v^{\otimes p})
    \Big]\,ds .
    \label{eq:peano-remainder-explicit}
\end{split}
\end{equation}
Let
$    P_{x,v}(s)\colon T_xM
       \longrightarrow T_{\gamma_{x,v}(s)}M $
denote parallel transport along $\gamma_{x,v}$. Since $\gamma_{x,v}$
is a geodesic,
\begin{equation}
    \dot\gamma_{x,v}(s)=P_{x,v}(s)v.
    \label{eq:velocity-parallel-transport}
\end{equation}
Consequently, the expression in square brackets in
\eqref{eq:peano-remainder-explicit} is
\[
\begin{split}
    &
    (\nabla^{(p)}f)_{\gamma_{x,v}(s)}
    \left(
        (P_{x,v}(s)v)^{\otimes p}
    \right)
    -
    (\nabla^{(p)}f)_x(v^{\otimes p}).
\end{split}
\]

For $v\neq0$, write $    v=r u,
    r=|v|, 
    |u|=1.$
By multilinearity, the preceding difference is equal to $r^p$ times
\[
\begin{split}
    &
    (\nabla^{(p)}f)_{\gamma_{x,ru}(s)}
    \left(
       (P_{x,ru}(s)u)^{\otimes p}
    \right)
    -
    (\nabla^{(p)}f)_x(u^{\otimes p}).
\end{split}
\]
The latter depends continuously on
$   (x,r,u,s)$
and vanishes when $r=0$. By compactness of $K$, of the unit sphere
bundle above $K$, and of $[0,1]$, this convergence to zero as
$r\downarrow0$ is uniform.
Hence there exists a function $\omega_{f,K}$ with
$\omega_{f,K}(r)\to 0$ as$r\downarrow0$
such that
\[
\begin{split}
    &\left|
    (\nabla^{(p)}f)_{\gamma_{x,v}(s)}
    \bigl(\dot\gamma_{x,v}(s)^{\otimes p}\bigr)
    -
    (\nabla^{(p)}f)_x(v^{\otimes p})
    \right|
    \leq
    p!\,\omega_{f,K}(|v|)|v|^p
\end{split}
\]
uniformly in $x\in K$ and $s\in[0,1]$. Integrating in
\eqref{eq:peano-remainder-explicit} and absorbing the fixed numerical
constant into $\omega_{f,K}$ yields
\[
    |R_{p,f}(x,v)|
    \leq
    \omega_{f,K}(|v|)|v|^p.
\]
This proves the result.
\end{proof}

Let $\mathcal U\subset M\times M$ be a neighborhood of the diagonal
on which the logarithmic increment
\[
    e_\nabla(x,y):=\exp_x^{-1}(y)
\]
is well defined. Thus $
    y=\exp_x(e_\nabla(x,y)).$
Theorem \ref{thm:covariant-taylor-peano} implies the following:
\begin{corollary}[Intrinsic Taylor expansion]
\label{cor:covariant-taylor-two-points}
Let $p\geq1$, $f\in C^p(M)$, and let $K\subset M$ be compact. For
$(x,y)\in {\mathcal U}\cap K^2$, we have
\begin{equation}
\begin{split}
    f(y)-f(x)
    &=
    \sum_{k=1}^p
    \frac1{k!}
    (\nabla^{(k)}f)_x
    \bigl(e_\nabla(x,y)^{\otimes k}\bigr)
    +
    R_{p,f}(x,y),
    \label{eq:covariant-taylor-two-points}
\end{split}
\end{equation}
where
\begin{equation}
    |R_{p,f}(x,y)|
    \leq
    \omega_{f,K}(|e_\nabla(x,y)|)
    |e_\nabla(x,y)|^p,\qquad \omega_{f,K}(r)\longrightarrow0
    \qquad\text{as }r\downarrow0.
    \label{eq:covariant-taylor-two-points-remainder}
\end{equation}
\end{corollary}

\section{Pathwise integration and change of variable formula}
\label{sec:change-of-variable}
We now derive an intrinsic change of variable formula for continuous
manifold-valued paths with finite $p-$th variation with arbitrary $p\in \mathbb{N}$, extending  the corresponding pathwise change of variable formula in Euclidean
space  \cite{CP2019,follmer1981}.

We use the notations and assumptions of Sections \ref{sec.pthvariation} and \ref{sec:covariant-taylor}.
Consider
$    X\in V_p(M,\pi)$
(Definition~\ref{def:p-variation}) with $p-$th variation tensor denoted
    $[X]_{\pi}^{(p)}
    \in
    \mathcal M
    \bigl([0,T];X^*\Sym^p(TM)\bigr).$
Since $X$ is continuous and $|\pi_n|\to0$, the increments
\begin{equation}
    e_\nabla
    \bigl(X_{t_i^n},X_{t_{i+1}^n}\bigr)
    \in T_{X_{t_i^n}}M .
    \label{eq:logarithmic-partition-increment}
\end{equation}
are
well-defined for $n$ sufficiently large  and satisfy
\begin{equation}
    \max_{0\leq i<N_n}|e_\nabla
    \bigl(X_{t_i^n},X_{t_{i+1}^n}\bigr)|
    \mathop{\longrightarrow}^{n\to\infty} 0.
    \label{eq:max-log-increment-zero}
\end{equation}


\subsection{Rough change of variable formula}

\begin{theorem}[Change of variable formula]
\label{thm:change-of-variable}
Let $p\in\mathbb N$, $p\geq 2$, $f\in C^p(M,\mathbb R)$, and let
$X\in V_p(M,\pi)$ be a continuous path with finite $p$-th
variation tensor $[X]_{\pi}^{(p)}$ along the sequence of partitions
$\pi$.
Then the  covariant left Taylor sums
\begin{equation}
    I_{\pi_n}^{\nabla,p}(f,X):=\sum_{i=0}^{N_n-1}
    \sum_{k=1}^{p-1}
    \frac1{k!}
    (\nabla^{(k)}f)_{X_{t_i^n}}
    \left(
      e_\nabla(X_{t_i^n},X_{t_{i+1}^n})^{\otimes k}
    \right) \label{eq:compensated-covariant-sum}
\end{equation}
converge as $n\to\infty$, to a limit which we denote
\begin{equation}
    \int_0^T
    \left\langle
        T_{\nabla}^{p-1}f(X_t),
        d_{\pi}^{\,p-1}X_t
    \right\rangle := \mathop{\lim}_{n\to\infty}  \sum_{i=0}^{N_n-1}
    \sum_{k=1}^{p-1}
    \frac1{k!}
    (\nabla^{(k)}f)_{X_{t_i^n}}
    \left( e_\nabla(X_{t_i^n},X_{t_{i+1}^n})^{\otimes k}
    \right)\label{eq:pathwiseintegral}
    \end{equation}
    and the following change of variable formula holds:
\begin{equation}
\boxed{
\begin{split}
    f(X_T)-f(X_0)
    &=
    \int_0^T
    \left\langle
        T_{\nabla}^{p-1}f(X_t),
        d_{\pi}^{\,p-1}X_t
    \right\rangle
    +
    \frac1{p!}
    \int_0^T
    \left\langle
        \nabla^{(p)}f(X_t),
        d[X]_{\pi,t}^{(p)}
    \right\rangle .
\end{split}}
    \label{eq:manifold-change-of-variable}
\end{equation}
\end{theorem}

\begin{proof}
 $K=X([0,T])$ is compact. By the covariant Taylor
formula (Theorem~\ref{thm:covariant-taylor-peano}) there exists a positive function $\omega_{f,K}:[0,\infty)\to \mathbb{R}_+$ with
\[
    \omega_{f,K}(r)\longrightarrow0
    \qquad\text{as }r\downarrow0
\]
such that, for all $x\in K$ and all sufficiently small
$v\in T_xM$,
\begin{equation}
\begin{split}
    f(\exp_xv)-f(x)
    &=
    \sum_{k=1}^{p}
    \frac1{k!}
    (\nabla^{(k)}f)_x(v^{\otimes k})
    +
    R_{p,f}(x,v),
    \label{eq:taylor-proof-change-variable}
\end{split}
\end{equation}
with
\begin{equation}
    |R_{p,f}(x,v)|
    \leq
    \omega_{f,K}(|v|)\,|v|^p.
    \label{eq:taylor-proof-remainder}
\end{equation}

Apply \eqref{eq:taylor-proof-change-variable} with
\[
    x=X_{t_i^n},
    \qquad
    v=\xi_i^n
      =e_\nabla(X_{t_i^n},X_{t_{i+1}^n}).
\]
Since
$X_{t_{i+1}^n}   = \exp_{X_{t_i^n}}(\xi_i^n),$
we obtain
\begin{equation}
\begin{split}
    f(X_{t_{i+1}^n})-f(X_{t_i^n})
    &=
    \sum_{k=1}^{p-1}
    \frac1{k!}
    (\nabla^{(k)}f)_{X_{t_i^n}}
    \bigl((\xi_i^n)^{\otimes k}\bigr)+
    \frac1{p!}
    (\nabla^{(p)}f)_{X_{t_i^n}}
    \bigl((\xi_i^n)^{\otimes p}\bigr)
    +
    R_i^n,
    \label{eq:taylor-each-increment}
\end{split}
\end{equation}
where $
    R_i^n
    :=
    R_{p,f}(X_{t_i^n},\xi_i^n).$
Summing \eqref{eq:taylor-each-increment} over $i$, the left-hand side
telescopes and gives
\begin{equation}
\begin{split}
    f(X_T)-f(X_0)
    &=
    I_{\pi_n}^{\nabla,p}(f,X)
    +
    \frac1{p!}
    \sum_{i=0}^{N_n-1}
    (\nabla^{(p)}f)_{X_{t_i^n}}
    \bigl((\xi_i^n)^{\otimes p}\bigr)
    +
    \sum_{i=0}^{N_n-1}R_i^n .
    \label{eq:telescope-change-variable}
\end{split}
\end{equation}
We now study the last two terms.
By definition of the $p$-th variation tensor,
\[
    \sum_{i=0}^{N_n-1}
    (\xi_i^n)^{\otimes p}\,\delta_{t_i^n}
    \stackrel{*}{\rightharpoonup}
    [X]_{\pi}^{(p)} .
\]
Since $
    t\longmapsto
    (\nabla^{(p)}f)_{X_t}$
is a continuous section of
$X^*\Sym^p(T^*M)$, it follows that
\begin{equation}
\begin{split}
    \sum_{i=0}^{N_n-1}
    (\nabla^{(p)}f)_{X_{t_i^n}}
    \bigl((\xi_i^n)^{\otimes p}\bigr)
    \longrightarrow
    \int_0^T
    \left\langle
        \nabla^{(p)}f(X_t),
        d[X]_{\pi,t}^{(p)}
    \right\rangle .
    \label{eq:pth-term-limit}
\end{split}
\end{equation}
To show that the  remainder vanishes, note that 
by \eqref{eq:max-log-increment-zero},$\rho_n=\max_{0\leq i<N_n}|\xi_i^n|\to 0.$
The estimate \eqref{eq:taylor-proof-remainder} gives
\[
\begin{split}
    \left|
        \sum_{i=0}^{N_n-1}R_i^n
    \right|
    &\leq
    \sum_{i=0}^{N_n-1}|R_i^n|
    \leq
    \omega_{f,K}(\rho_n)
    \sum_{i=0}^{N_n-1}|\xi_i^n|^p .
    \label{eq:sum-remainder-bound}
\end{split}
\]
The path has finite $p$-energy along $\pi$, hence
\[
    \sup_n
    \sum_{i=0}^{N_n-1}|\xi_i^n|^p
    <\infty.
\]
Since $\omega_{f,K}(\rho_n)\to0$, we conclude that
\begin{equation}
    \sum_{i=0}^{N_n-1}R_i^n
    \longrightarrow 0.
    \label{eq:sum-remainder-zero}
\end{equation}
Passing to the limit in
\eqref{eq:telescope-change-variable} and using
\eqref{eq:pth-term-limit}--\eqref{eq:sum-remainder-zero} shows that
$I_{\pi_n}^{\nabla,p}(f,X)$ converges and 
\eqref{eq:manifold-change-of-variable} holds.
\end{proof}

\begin{remark}[The case $p=2$]
\label{rem:p2-follmer}
For $p=2$, $T_\nabla^1f=df,$
and the first term becomes
\begin{equation}
    \int_0^T df(X_t)\,d_\pi^\nabla X_t
    :=
    \lim_{n\to\infty}
    \sum_{i=0}^{N_n-1}
    df_{X_{t_i^n}}
    \left(
        e_\nabla(X_{t_i^n},X_{t_{i+1}^n})
    \right).
    \label{eq:manifold-follmer-integral}
\end{equation}
The change of variable formula reads
\begin{equation}
\boxed{
    f(X_T)-f(X_0)
    =
    \int_0^T df(X_t)\,d_\pi^\nabla X_t
    +
    \frac12
    \int_0^T
    \left\langle
        \nabla^{(2)}f(X_t),
        d[X]_{\pi,t}^{(2)}
    \right\rangle .
}
    \label{eq:manifold-follmer-ito}
\end{equation}
In Euclidean space this is precisely F\"ollmer's pathwise It\^o
formula \cite{follmer1981}.
\end{remark}
\begin{remark}[Exact differential forms]
\label{rem:exact-form-integration}
Theorem~\ref{thm:change-of-variable} defines, for every exact
one-form $df$, a pathwise integral defined as the limit of left Riemann-Taylor sums \eqref{eq:compensated-covariant-sum}  along the partition sequence $\pi$. The integral
\[
    \int_0^T
    \left\langle
        T_{\nabla}^{p-1}f(X_t),
        d_{\pi}^{\,p-1}X_t
    \right\rangle
\]
is a manifold extension of the construction in \cite{CP2019} and
should  be viewed as an intrinsic higher-order analogue of
the F\"ollmer integral of the exact form $df$ along $X$. Note that, as in \cite{CP2019} the integral, defined as the limit of {\it left} Riemann-Taylor sums, is an It\^o-type integral.
\end{remark}

\begin{remark}[Dependence on the connection]
\label{rem:connection-dependence}
The $p$-th variation tensor $[X]_{\pi}^{(p)}$ is intrinsic in the
sense established in the previous section: its definition is
independent of the choice of admissible increment map.
The decomposition
\eqref{eq:manifold-change-of-variable}, however, uses the connection.
Both $\nabla^{(p)}f$
and the pathwise integral
\[
    \int
    \left\langle
        T_{\nabla}^{p-1}f,
        d_{\pi}^{\,p-1}X
    \right\rangle
\]
depend on $\nabla$. Their combination is independent of this choice, since it is equal to the intrinsic quantity
$f(X_T)-f(X_0)$. We will further comment on transformation under change of connection in Section \ref{subsec:connection-transformation}.
\end{remark}

\begin{remark}{\em 
For $X\in V_p(M,\pi)$, the sums
\[
    \sum_i
    (\nabla^{(k)}f)_{X_{t_i^n}}
    \bigl((\xi_i^n)^{\otimes k}\bigr),
    \qquad k<p,
\]
need not converge individually in general. The change of variable
formula shows that the particular linear combination
\[
    \sum_{k=1}^{p-1}\frac1{k!}
    (\nabla^{(k)}f)(\xi^{\otimes k})
\]
 does converge after
summation over the partition. This compensation mechanism also appears in \cite{CP2019} and in rough integration theory \cite{FrizHairer} but in contrast to rough integration the term of order $p$ is excluded from the sum. Note also that in general this limit depends on the partition sequence $\pi$. Invariance with respect to $\pi$ under roughness conditions for $X$ has been studied in \cite{das2023} for the vector case for $p=2$.}
\end{remark}
\subsection{Integration of closed one-forms}
\label{subsec:closed-one-forms}

The preceding construction for exact one-forms extends naturally to
closed one-forms. The resulting change-of-variable formula contains a
topological endpoint term, which may depend on the homotopy class of
the path.

Let $\alpha\in C^{p-1}(T^*M)$ be a closed one-form. For
$1\leq k\leq p$, define
\begin{equation}
 \mathcal D_{\nabla}^{(k)}\alpha
 :=
 \operatorname{Sym}\bigl(\nabla^{k-1}\alpha\bigr)
 \in
 \Gamma\bigl(\operatorname{Sym}^k(T^*M)\bigr),
 \label{eq:closed-form-covariant-derivatives}
\end{equation}
where $\nabla^0\alpha:=\alpha$. Thus
$\mathcal D_{\nabla}^{(1)}\alpha=\alpha$. If $\alpha=df$ locally,
then
\begin{equation}
 \mathcal D_{\nabla}^{(k)}\alpha
 =
 \nabla^{(k)}f,
 \qquad 1\leq k\leq p.
 \label{eq:local-potential-derivatives}
\end{equation}

We denote the corresponding truncated covariant jet by
\[
 T_\nabla^{p-1}\alpha
 :=
 \left(
 \alpha,
 \frac1{2!}\mathcal D_\nabla^{(2)}\alpha,
 \ldots,
 \frac1{(p-1)!}\mathcal D_\nabla^{(p-1)}\alpha
 \right).
\]
For
\[
 \xi_i^n
 :=
 \exp_{X_{t_i^n}}^{-1}(X_{t_{i+1}^n}),
\]
consider the covariant left Taylor sums
\begin{equation}
 I_{\pi,n}^{\nabla,p}(\alpha;X)
 :=
 \sum_{i=0}^{N_n-1}
 \sum_{k=1}^{p-1}
 \frac1{k!}
 \mathcal D_\nabla^{(k)}\alpha(X_{t_i^n})
 \bigl((\xi_i^n)^{\otimes k}\bigr).
 \label{eq:closed-form-discrete-sum}
\end{equation}

To describe their limit, assume without loss of generality that $M$
is connected and let
\[
 q:\widetilde M\longrightarrow M
\]
be its universal covering. Since $q^*\alpha$ is closed and
$\widetilde M$ is simply connected, there exists
$F\in C^p(\widetilde M,\mathbb R)$ such that
\[
 dF=q^*\alpha.
\]
For a continuous path $X:[0,T]\to M$, choose a lift
$\widetilde X:[0,T]\to\widetilde M$ and define
\begin{equation}
 \mathcal A_\alpha(X)
 :=
 F(\widetilde X_T)-F(\widetilde X_0).
 \label{eq:closed-form-topological-increment}
\end{equation}
This quantity is independent of the choices of $F$ and
$\widetilde X$. It will be called the topological increment of
$\alpha$ along $X$.

\begin{proposition}[Change of variable formula for closed one-forms]
\label{prop:closed-one-form-integral}
Let $p\geq2$, let $X\in V_p(M,\pi)$, and let
$\alpha\in C^{p-1}(T^*M)$ be closed. Then the sums
\eqref{eq:closed-form-discrete-sum} converge. Writing
\[
 \int_0^T
 \left\langle
 T_\nabla^{p-1}\alpha(X_t),
 d_\pi^{p-1}X_t
 \right\rangle
 :=
 \lim_{n\to\infty}
 I_{\pi,n}^{\nabla,p}(\alpha;X),
\]
one has
\begin{equation}
\begin{split}
 \mathcal A_\alpha(X)
 &=
 \int_0^T
 \left\langle
 T_\nabla^{p-1}\alpha(X_t),
 d_\pi^{p-1}X_t
 \right\rangle +
 \frac1{p!}
 \int_0^T
 \left\langle
 \mathcal D_\nabla^{(p)}\alpha(X_t),
 d[X]_{\pi,t}^{(p)}
 \right\rangle .
\end{split}
\label{eq:closed-form-change-variable}
\end{equation}
Equivalently,
\begin{equation}
\begin{split}
 \int_0^T
 \left\langle
 T_\nabla^{p-1}\alpha(X_t),
 d_\pi^{p-1}X_t
 \right\rangle
 &=
 \mathcal A_\alpha(X)
 -
 \frac1{p!}
 \int_0^T
 \left\langle
 \mathcal D_\nabla^{(p)}\alpha(X_t),
 d[X]_{\pi,t}^{(p)}
 \right\rangle .
\end{split}
\label{eq:closed-form-integral-representation}
\end{equation}
If $\alpha=df$ is globally exact, then
\[
 \mathcal A_\alpha(X)=f(X_T)-f(X_0),
\]
and~\eqref{eq:closed-form-change-variable} reduces to
Theorem~\ref{thm:change-of-variable}.
Moreover, if $\nabla$ and $\widetilde\nabla$ are two affine
connections, then
\begin{equation}
\begin{split}
&
 \int_0^T
 \left\langle
 T_{\widetilde\nabla}^{p-1}\alpha(X_t),
 d_{\pi,\widetilde\nabla}^{p-1}X_t
 \right\rangle
 -
 \int_0^T
 \left\langle
 T_\nabla^{p-1}\alpha(X_t),
 d_{\pi,\nabla}^{p-1}X_t
 \right\rangle
\\
&\qquad =
 \frac1{p!}
 \int_0^T
 \left\langle
 \mathcal D_\nabla^{(p)}\alpha(X_t)
 -
 \mathcal D_{\widetilde\nabla}^{(p)}\alpha(X_t),
 d[X]_{\pi,t}^{(p)}
 \right\rangle .
\end{split}
\label{eq:closed-form-connection-change}
\end{equation}
\end{proposition}

\begin{proof}
We first verify that
$\mathcal A_\alpha(X)$ is well-defined. Any two primitives of
$q^*\alpha$ differ by a constant because $\widetilde M$ is
connected, so the difference
$F(\widetilde X_T)-F(\widetilde X_0)$ does not depend on the chosen
primitive.

Suppose next that $\widetilde X'$ is another lift of $X$. Since
$q:\widetilde M\to M$ is a universal covering, there is a deck
transformation $G:\widetilde M\to\widetilde M$ such that
\[
 \widetilde X'=G\circ\widetilde X.
\]
Because $q\circ G=q$, one has
\[
 d(F\circ G)
 =
 G^*dF
 =
 G^*q^*\alpha
 =
 (q\circ G)^*\alpha
 =
 q^*\alpha
 =
 dF.
\]
It follows that $F\circ G-F$ is constant. Consequently,
\[
\begin{split}
 F(\widetilde X'_T)-F(\widetilde X'_0)
 &=
 F(G(\widetilde X_T))-F(G(\widetilde X_0))
=
 F(\widetilde X_T)-F(\widetilde X_0).
\end{split}
\]
Thus~\eqref{eq:closed-form-topological-increment} is independent of
the lift.

We now lift the analytic construction to $\widetilde M$. The
connection $\nabla$ induces a unique affine connection
$\widetilde\nabla$ on $\widetilde M$ for which $q$ is an affine local
diffeomorphism. Equivalently, for vector fields related by $q$,
\[
 Dq\bigl(\widetilde\nabla_{\widetilde U}\widetilde V\bigr)
 =
 \nabla_{Dq(\widetilde U)}Dq(\widetilde V).
\]
In particular, the exponential maps satisfy
\begin{equation}
 q\bigl(
 \exp_{\widetilde x}^{\widetilde\nabla}(\widetilde v)
 \bigr)
 =
 \exp_{q(\widetilde x)}^\nabla
 \bigl(Dq_{\widetilde x}\widetilde v\bigr)
\label{eq:covering-exponential-compatibility}
\end{equation}
whenever both sides are defined.
Let
\[
 \widetilde\xi_i^n
 :=
 \left(
 \exp_{\widetilde X_{t_i^n}}^{\widetilde\nabla}
 \right)^{-1}
 \bigl(\widetilde X_{t_{i+1}^n}\bigr).
\]
Since $\widetilde X([0,T])$ is compact, it can be covered by finitely
many neighborhoods on which $q$ is a diffeomorphism and on which the
local exponential maps are injective. Uniform continuity of
$\widetilde X$ then implies that, for all sufficiently large $n$,
the points $\widetilde X_{t_i^n}$ and
$\widetilde X_{t_{i+1}^n}$ lie in a common such neighborhood.
Equation~\eqref{eq:covering-exponential-compatibility} therefore
gives
\begin{equation}
 Dq_{\widetilde X_{t_i^n}}\widetilde\xi_i^n
 =
 \xi_i^n
 \label{eq:lifted-logarithmic-increments}
\end{equation}
for every $i$ and every sufficiently large $n$.

Equip $\widetilde M$ with the lifted auxiliary metric
$\widetilde g=q^*g$. Since $Dq$ is then a fiberwise isometry,
\eqref{eq:lifted-logarithmic-increments} implies
\[
 \sum_i|\widetilde\xi_i^n|_{\widetilde g}^p
 =
 \sum_i|\xi_i^n|_g^p.
\]
Thus $\widetilde X$ has finite $p$-energy along $\pi$.
Moreover, define the continuous bundle isomorphism along the lifted
path by
\[
 L_t
 :=
 \bigl(Dq_{\widetilde X_t}^{-1}\bigr)^{\otimes p}
 :
 \operatorname{Sym}^p(T_{X_t}M)
 \longrightarrow
 \operatorname{Sym}^p(T_{\widetilde X_t}\widetilde M).
\]
For sufficiently large $n$,
\[
 \sum_i(\widetilde\xi_i^n)^{\otimes p}\delta_{t_i^n}
 =
 L_\cdot
 \left(
 \sum_i(\xi_i^n)^{\otimes p}\delta_{t_i^n}
 \right).
\]
Since $L_\cdot$ is continuous and the tensor measures on the
right converge weakly to $[X]_\pi^{(p)}$, it follows that
$\widetilde X\in V_p(\widetilde M,\pi)$ and
\begin{equation}
 d[\widetilde X]_{\pi,t}^{(p)}
 =
 \bigl(Dq_{\widetilde X_t}^{-1}\bigr)^{\otimes p}
 d[X]_{\pi,t}^{(p)}.
 \label{eq:lifted-variation-tensor}
\end{equation}

Because $dF=q^*\alpha$ and $q$ is affine with respect to the lifted
connection, covariant differentiation commutes with pullback:
\[
 \widetilde\nabla^{\,k-1}dF
 =
 q^*(\nabla^{k-1}\alpha).
\]
After symmetrization, this gives
\begin{equation}
 \widetilde\nabla^{(k)}F
 =
 q^*\mathcal D_\nabla^{(k)}\alpha,
 \qquad 1\leq k\leq p.
 \label{eq:lifted-covariant-derivatives}
\end{equation}
Combining~\eqref{eq:lifted-logarithmic-increments} and
\eqref{eq:lifted-covariant-derivatives}, we obtain, for every
sufficiently large $n$,
\begin{equation}
\begin{split}
&
 \sum_i\sum_{k=1}^{p-1}
 \frac1{k!}
 \widetilde\nabla^{(k)}F(\widetilde X_{t_i^n})
 \bigl((\widetilde\xi_i^n)^{\otimes k}\bigr)
\\
&\qquad =
 \sum_i\sum_{k=1}^{p-1}
 \frac1{k!}
 \mathcal D_\nabla^{(k)}\alpha(X_{t_i^n})
 \bigl((\xi_i^n)^{\otimes k}\bigr)
 =
 I_{\pi,n}^{\nabla,p}(\alpha;X).
\end{split}
\label{eq:lifted-sums-equal}
\end{equation}

Theorem~\ref{thm:change-of-variable} may now be applied to $F$ and $\widetilde X$. It shows
that the left-hand side of~\eqref{eq:lifted-sums-equal} converges and
that
\[
\begin{split}
 F(\widetilde X_T)-F(\widetilde X_0)
 &=
 \lim_{n\to\infty}I_{\pi,n}^{\nabla,p}(\alpha;X)
 +
 \frac1{p!}
 \int_0^T
 \left\langle
 \widetilde\nabla^{(p)}F(\widetilde X_t),
 d[\widetilde X]_{\pi,t}^{(p)}
 \right\rangle .
\end{split}
\]
By~\eqref{eq:lifted-covariant-derivatives},
\eqref{eq:lifted-variation-tensor}, and duality of fibers
\[
 \left\langle
 \widetilde\nabla^{(p)}F(\widetilde X_t),
 d[\widetilde X]_{\pi,t}^{(p)}
 \right\rangle
 =
 \left\langle
 \mathcal D_\nabla^{(p)}\alpha(X_t),
 d[X]_{\pi,t}^{(p)}
 \right\rangle .
\]
Substituting this identity and the definition of
$\mathcal A_\alpha(X)$ proves
\eqref{eq:closed-form-change-variable} and
\eqref{eq:closed-form-integral-representation}.

If $\alpha=df$ globally, then $F-f\circ q$ is constant, and hence
\[
 \mathcal A_\alpha(X)
 =
 f(X_T)-f(X_0).
\]
This proves the exact-form case.
Finally, apply~\eqref{eq:closed-form-integral-representation} once
with $\nabla$ and once with $\widetilde\nabla$. The topological
increment $\mathcal A_\alpha(X)$ and the variation tensor
$[X]_\pi^{(p)}$ do not depend on the connection. Subtracting the two
representations gives~\eqref{eq:closed-form-connection-change}.
\end{proof}
The topological increment  admits a  local description.
If
$
 0=s_0<s_1<\cdots<s_m=T$
and the path segments
$X([s_{j-1},s_j])$ lie in open sets $U_j$ on which
$\alpha=df_j$, then
\[
 \mathcal A_\alpha(X)
 =
 \sum_{j=1}^m
 \bigl(
 f_j(X_{s_j})-f_j(X_{s_{j-1}})
 \bigr).
\]
The value is independent of the partition and of the local
potentials. For a closed loop it is the usual period of $\alpha$ and
may be nonzero. 
\subsection{Uniform convergence and $p$-th variation identities}

\begin{lemma}[Uniform localization]
\label{lem:uniform-localization}
Let $E\to[0,T]$ be a finite-dimensional normed vector bundle, and let
$\mu_n,\mu\in\mathcal M([0,T];E)$ be finite vector measures such that
$\mu_n\stackrel{*}{\rightharpoonup}\mu.$
Suppose that there are finite positive measures $\nu_n,\nu$ such that
\[
        |\mu_n|\leq \nu_n,
        \qquad
        \nu_n\rightharpoonup\nu,
\]
and assume that $\nu$ has no atoms. Then, for every continuous section
$A\in C([0,T];E^*)$,
\begin{equation}
 \sup_{0\leq t\leq T}
 \left|
   \int_{[0,t)}\langle A_s,d\mu_n(s)\rangle
   -
   \int_{[0,t]}\langle A_s,d\mu(s)\rangle
 \right|
 \longrightarrow 0.
 \label{eq:uniform-localization}
\end{equation}
Moreover, $ |\mu|\leq \nu,$
so $\mu$ has no atoms.
\end{lemma}

\begin{proof}
Define finite signed measures $\eta_n$ and $\eta$ by
\[
 d\eta_n(s):=\langle A_s,d\mu_n(s)\rangle,
 \qquad
 d\eta(s):=\langle A_s,d\mu(s)\rangle.
\]
Since $\varphi A$ is a continuous section of $E^*$ for every
$\varphi\in C([0,T])$, weak-* convergence of $\mu_n$ gives $\eta_n\rightharpoonup\eta.$
Furthermore,
\begin{equation}         |\eta_n|
          \leq \|A\|_\infty |\mu_n|
          \leq \|A\|_\infty\nu_n.
 \label{eq:eta-domination}
\end{equation}
Passing to the limit in
\[
 \left|\int_{[0,T]}\varphi\,d\eta_n\right|
 \leq
 \|A\|_\infty\int_{[0,T]}|\varphi|\,d\nu_n
\]
shows that
\[
 \left|\int_{[0,T]}\varphi\,d\eta\right|
 \leq
 \|A\|_\infty\int_{[0,T]}|\varphi|\,d\nu
\]
for every continuous $\varphi$. Hence
\begin{equation}         |\eta|\leq\|A\|_\infty\nu.
 \label{eq:eta-limit-domination}
\end{equation}
Applying this argument locally in bundle trivializations and then
taking the supremum over continuous dual sections of norm at most one
gives $|\mu|\leq\nu$. In particular, both $\mu$ and $\eta$ are
atomless.

Fix $t\in[0,T]$. Since $\nu(\{t\})=0$, the indicator of $[0,t)$ can
be approximated from above and below by continuous functions whose
difference is supported in an arbitrarily small neighborhood of $t$.
The errors in the corresponding integrals against $\eta_n$ are,
by \eqref{eq:eta-domination}, bounded by $\|A\|_\infty$ times the
$\nu_n$-mass of this neighborhood. Weak convergence of the positive
measures $\nu_n$ and atomlessness of $\nu$ therefore yield
\begin{equation}
                    \eta_n([0,t))\longrightarrow\eta([0,t])
                    \qquad\text{for every }t\in[0,T].
 \label{eq:pointwise-cumulative-convergence}
\end{equation}
It remains to make the convergence uniform. Let $\varepsilon>0$.
Since $\nu$ is finite and atomless, there exists a partition $
                 0=s_0<s_1<\cdots<s_m=T$
such that
\begin{equation}
                    \nu([s_j,s_{j+1}])<\varepsilon,
                    \qquad 0\leq j<m.
 \label{eq:small-measure-subdivision}
\end{equation}
Every interval $[s_j,s_{j+1}]$ is a continuity set for $\nu$, and
hence
\[
       \nu_n([s_j,s_{j+1}])
       \longrightarrow
       \nu([s_j,s_{j+1}]).
\]
For $s_j\leq t\leq s_{j+1}$, domination gives
\[
\begin{split}
 \left|\eta_n([0,t))-\eta([0,t])\right|
 &\leq
 \left|\eta_n([0,s_j))-\eta([0,s_j])\right|
 +\|A\|_\infty\nu_n([s_j,s_{j+1}])
 +\|A\|_\infty\nu([s_j,s_{j+1}]).
\end{split}
\]
Since there are only finitely many partition points,
\eqref{eq:pointwise-cumulative-convergence} and
\eqref{eq:small-measure-subdivision} imply
\[
 \limsup_{n\to\infty}
 \sup_{0\leq t\leq T}
 \left|\eta_n([0,t))-\eta([0,t])\right|
 \leq 2\|A\|_\infty\varepsilon.
\]
Letting $\varepsilon\downarrow0$ proves
\eqref{eq:uniform-localization}.
\end{proof}

\begin{proposition}[Uniform convergence and transformation of
$p$-th variation]
\label{prop:uniform-p-variation}
Let $p\geq2$ be an integer, $X\in V_p(M,\pi)$, $\nabla$ 
an affine connection on $M$ and  $f\in C^p(M,\mathbb R)$. Set
\[
 \xi_i^n
 :=
 e^\nabla(X_{t_i^n},X_{t_{i+1}^n}),
 \qquad
 \mu_n
 :=
 \sum_{i=0}^{N_n-1}
 (\xi_i^n)^{\otimes p}\delta_{t_i^n}.
\]
Fix an auxiliary Riemannian metric $g$ and define the 
$p$-energy measures
\begin{equation}
 \nu_n
 :=
 \sum_{i=0}^{N_n-1}
 |\xi_i^n|_g^p\delta_{t_i^n}.
 \label{eq:absolute-energy-measures}
\end{equation}
Assume one of the following conditions:
\begin{enumerate}
\item
$p$ is even and $[X]_\pi^{(p)}$ has no atoms;
\item
$p$ is odd and $\nu_n$ converges weakly to a finite atomless
positive measure $\nu$.
\end{enumerate}
Then,  for $t\in[0,T]$, the stopped covariant Taylor sums
\begin{equation}
\begin{split}
 I_n^\nabla(t;f,X)
 :=
 \sum_{\substack{0\leq i<N_n\\t_i^n<t}}
 \sum_{k=1}^{p-1}\frac1{k!}
 (\nabla^{(k)}f)_{X_{t_i^n}}
 \left(
 e^\nabla
 \bigl(
 X_{t_i^n},X_{t_{i+1}^n\wedge t}
 \bigr)^{\otimes k}
 \right).
\end{split}
\label{eq:stopped-covariant-taylor-sums}
\end{equation}
converge uniformly on $[0,T]$ to
the continuous path
\begin{equation}
\begin{split}
 I_t^\nabla(f,X)
 &=
 f(X_t)-f(X_0)
 -
 \frac1{p!}
 \int_{[0,t]}
 \left\langle
 \nabla^{(p)}f(X_s),d[X]_{\pi,s}^{(p)}
 \right\rangle .
\end{split}
\label{eq:uniform-integral-representation}
\end{equation}
In particular,
\begin{equation}
 \sup_{0\leq t\leq T}
 \left|
 I_n^\nabla(t;f,X)-I_t^\nabla(f,X)
 \right|
 \longrightarrow0.
 \label{eq:uniform-integral-convergence}
\end{equation}

Moreover, $I^\nabla(f,X)\in V_p(\mathbb R,\pi)$ and
\begin{equation}
 d[I^\nabla(f,X)]_{\pi,t}^{(p)}
 =
 \left\langle
 (df_{X_t})^{\otimes p},d[X]_{\pi,t}^{(p)}
 \right\rangle .
\label{eq:p-variation-transformation}
\end{equation}
The convergence underlying \eqref{eq:p-variation-transformation} is
uniform in time:
\begin{equation}
\begin{split}
 \sup_{0\leq t\leq T}
 \Bigg|
 &
 \sum_{i=0}^{N_n-1}
 \left(
 I_{t_{i+1}^n\wedge t}^\nabla(f,X)
 -
 I_{t_i^n\wedge t}^\nabla(f,X)
 \right)^p
 -
 \int_{[0,t]}
 \left\langle
 (df_{X_s})^{\otimes p},d[X]_{\pi,s}^{(p)}
 \right\rangle
 \Bigg|
 \longrightarrow0.
\end{split}
\label{eq:uniform-p-variation-transformation}
\end{equation}
When $p$ is odd, the powers and the limiting measure in
\eqref{eq:p-variation-transformation} are signed.
\end{proposition}

\begin{proof}
We first show that in both cases the tensor measures $\mu_n$ are
dominated by positive measures converging to an atomless limit.

Suppose that $p=2m$ is even. Define the continuous covariant
$p$-tensor field along $X$ by
\[
                      G_t:=\operatorname{Sym}(g_{X_t}^{\otimes m}).
\]
For every $v\in T_{X_t}M$,
\[
                         G_t(v^{\otimes p})=|v|_g^p.
\]
Consequently, for every $\varphi\in C([0,T])$,
\[
 \int_{[0,T]}\varphi\,d\nu_n
 =
 \int_{[0,T]}\langle\varphi G,d\mu_n\rangle
 \longrightarrow
 \int_{[0,T]}
 \left\langle
 \varphi G,d[X]_\pi^{(p)}
 \right\rangle.
\]
Thus
\begin{equation}
 \nu_n\rightharpoonup\nu,
 \qquad
 d\nu(t):=
 \left\langle G_t,d[X]_{\pi,t}^{(p)}\right\rangle.
 \label{eq:even-energy-limit}
\end{equation}
The measure $\nu$ is positive because it is a weak limit of positive
measures. It is atomless because $[X]_\pi^{(p)}$ is atomless.

For odd $p$, the same conclusion is assumed. Since
$ |(\xi_i^n)^{\otimes p}|=|\xi_i^n|_g^p $
we have in both cases
\begin{equation}
 \mu_n\stackrel{*}{\rightharpoonup}[X]_\pi^{(p)},
 \qquad
 |\mu_n|=\nu_n\rightharpoonup\nu,
 \qquad
 \nu\ \text{is atomless}.
 \label{eq:common-domination}
\end{equation}

Let
\[
 \operatorname{osc}(X,\pi_n)
 :=
 \max_{0\leq i<N_n}
 \sup_{u,v\in[t_i^n,t_{i+1}^n]}d_g(X_u,X_v).
\]
Uniform continuity of $X$ gives
\[
                    \operatorname{osc}(X,\pi_n)\longrightarrow0.
\]
Admissible increments are comparable with Riemannian distance on
the compact trace of $X$. Hence, for all sufficiently large $n$,
\begin{equation}
 \sup_{\substack{0\leq t\leq T\\t_i^n<t}}
 \left|
 e^\nabla
 \bigl(
 X_{t_i^n},X_{t_{i+1}^n\wedge t}
 \bigr)
 \right|_g
 \leq
 C\operatorname{osc}(X,\pi_n).
 \label{eq:stopped-increment-bound}
\end{equation}

Define
\[
\begin{split}
 B_n(t)
 :=
 \sum_{\substack{0\leq i<N_n\\t_i^n<t}}
 (\nabla^{(p)}f)_{X_{t_i^n}}
 \left(
 e^\nabla
 \bigl(
 X_{t_i^n},X_{t_{i+1}^n\wedge t}
 \bigr)^{\otimes p}
 \right)
\end{split}
\]
and
\[
 B(t)
 :=
 \int_{[0,t]}
 \left\langle
 \nabla^{(p)}f(X_s),d[X]_{\pi,s}^{(p)}
 \right\rangle.
\]
Lemma~\ref{lem:uniform-localization}, applied with
\[
 A_s=\nabla^{(p)}f(X_s),
\]
shows that
\[
 \sup_{0\leq t\leq T}
 \left|
 \sum_{t_i^n<t}
 (\nabla^{(p)}f)_{X_{t_i^n}}
 \bigl((\xi_i^n)^{\otimes p}\bigr)
 -
 B(t)
 \right|
 \longrightarrow0.
\]
For a fixed $t$, the sum in this display and $B_n(t)$ differ in at
most one term. By \eqref{eq:stopped-increment-bound}, that difference
is bounded uniformly in $t$ by
\[
 C_f\operatorname{osc}(X,\pi_n)^p.
\]
It follows that
\begin{equation}
             \sup_{0\leq t\leq T}|B_n(t)-B(t)|
             \longrightarrow0.
 \label{eq:uniform-highest-order-term}
\end{equation}
The scalar measure defining $B$ is dominated by
$\|\nabla^{(p)}f\|_\infty\nu$. Since $\nu$ is atomless, $B$ is a
continuous finite-variation path.
Apply the covariant Taylor formula to every stopped increment. Summing across the partition gives
\begin{equation}
 f(X_t)-f(X_0)
 =
I_n^\nabla(t;f,X)+\frac1{p!}B_n(t)+R_n(t),
 \label{eq:stopped-taylor-decomposition}
\end{equation}
where
\[
 |R_n(t)|
 \leq
 \omega_{f,K}
 \bigl(C\operatorname{osc}(X,\pi_n)\bigr)
 \sum_{t_i^n<t}
 \left|
 e^\nabla
 \bigl(
 X_{t_i^n},X_{t_{i+1}^n\wedge t}
 \bigr)
 \right|_g^p.
\]
All but at most one of the terms in the last sum are full partition
increments. Therefore,
\[
\begin{split}
 \sup_{0\leq t\leq T}
 \sum_{t_i^n<t}
 \left|
 e^\nabla
 \bigl(
 X_{t_i^n},X_{t_{i+1}^n\wedge t}
 \bigr)
 \right|_g^p
 \leq
 \nu_n([0,T])
 +
 C^p\operatorname{osc}(X,\pi_n)^p.
\end{split}
\]
The right-hand side is uniformly bounded, while
$
 \omega_{f,K}
 \bigl(C\operatorname{osc}(X,\pi_n)\bigr)
 \longrightarrow0.$
Hence
\[
                     \sup_{0\leq t\leq T}|R_n(t)|
                     \longrightarrow0.
\]
Together with \eqref{eq:uniform-highest-order-term} and
\eqref{eq:stopped-taylor-decomposition}, this proves
\eqref{eq:uniform-integral-representation} and
\eqref{eq:uniform-integral-convergence}.

We next identify the $p$-th variation of the limiting integral. Set
\[
 Y_t:=f(X_t),
 \qquad
 A_t:=\frac1{p!}B(t).
\]
Then
\begin{equation}
                     I_t^\nabla(f,X)=Y_t-Y_0-A_t.
 \label{eq:integral-as-fv-correction}
\end{equation}
By Proposition~2.11, $Y\in V_p(\mathbb R,\pi)$ and
\begin{equation}
 d[Y]_{\pi,t}^{(p)}
 =
 \left\langle
 (df_{X_t})^{\otimes p},d[X]_{\pi,t}^{(p)}
 \right\rangle.
 \label{eq:variation-of-fX}
\end{equation}

The path $A$ is continuous and has finite variation. Writing
\[
 \Delta_i^nA:=A_{t_{i+1}^n}-A_{t_i^n},
\]
we obtain
\[
 \sum_i|\Delta_i^nA|^p
 \leq
 \left(\max_i|\Delta_i^nA|\right)^{p-1}
 \operatorname{Var}(A;[0,T])
 \longrightarrow0.
 \label{eq:fv-zero-p-variation}
\]
Similarly, set
\[
 \Delta_i^nY:=Y_{t_{i+1}^n}-Y_{t_i^n},
 \qquad
 \Delta_i^nI:=
 I_{t_{i+1}^n}^\nabla(f,X)-I_{t_i^n}^\nabla(f,X).
\]
By \eqref{eq:integral-as-fv-correction},
\[
                       \Delta_i^nI=\Delta_i^nY-\Delta_i^nA.
\]
For every integer $p\geq2$,
\[
 |(a-b)^p-a^p|
 \leq
 C_p\sum_{k=1}^p|a|^{p-k}|b|^k.
\]
Consequently, H\"older's inequality gives
\[
\begin{split}
 \sum_i
 \left|
 (\Delta_i^nI)^p-(\Delta_i^nY)^p
 \right|
 &\leq
 C_p\sum_{k=1}^p
 \left(\sum_i|\Delta_i^nY|^p\right)^{(p-k)/p}
 \left(\sum_i|\Delta_i^nA|^p\right)^{k/p}
\longrightarrow0.
\end{split}
\label{eq:variation-measure-tv-comparison}
\]
Here the first factors are uniformly bounded because $Y$ has finite
$p$-energy, 
while the second factors converge to zero by the preceding
finite-variation estimate.
It follows that the discrete $p$-th variation measures of
$I^\nabla(f,X)$ and $Y$ differ in total variation by a quantity
tending to zero. Moreover,
\[
 \sum_i|\Delta_i^nI|^p
 \leq
 2^{p-1}
 \left(
   \sum_i|\Delta_i^nY|^p
   +
   \sum_i|\Delta_i^nA|^p
 \right),
\]
so $I^\nabla(f,X)$ has finite $p$-energy. Therefore
\[
             [I^\nabla(f,X)]_\pi^{(p)}=[Y]_\pi^{(p)}.
\]
Combining this identity with \eqref{eq:variation-of-fX} proves
\eqref{eq:p-variation-transformation}. Notice that this argument does
not use the parity of $p$.
It remains to prove uniform convergence of the cumulative variation
sums. Let
\[
 \eta_n^Y
 :=
 \sum_{i=0}^{N_n-1}
 (\Delta_i^nY)^p\delta_{t_i^n}.
\]
By Proposition~2.11,
$ \eta_n^Y\rightharpoonup[Y]_\pi^{(p)}.$
Since $f$ is $C^1$ and $X([0,T])$ is compact, there is a constant
$C_f$ such that, for all sufficiently large $n$,
\[
                       |\Delta_i^nY|\leq C_f|\xi_i^n|_g.
\]
It follows that
\[
                         |\eta_n^Y|\leq C_f^p\nu_n.
\]
Lemma~\ref{lem:uniform-localization} therefore gives
\begin{equation}
 \sup_{0\leq t\leq T}
 \left|
 \sum_{t_i^n<t}(\Delta_i^nY)^p
 -
 [Y]_\pi^{(p)}([0,t])
 \right|
 \longrightarrow0.
 \label{eq:uniform-cumulative-Y-variation}
\end{equation}
On the other hand, the preceding total-variation estimate implies
\[
\begin{split}
 \sup_{0\leq t\leq T}
 \left|
 \sum_{t_i^n<t}
 \left(
   (\Delta_i^nI)^p-(\Delta_i^nY)^p
 \right)
 \right|
 &\leq
 \sum_i
 \left|
   (\Delta_i^nI)^p-(\Delta_i^nY)^p
 \right|
\\
 &\longrightarrow0.
\end{split}
\]
Combining this with \eqref{eq:uniform-cumulative-Y-variation} gives
uniform convergence for sums of full increments.

The stopped sum in
\eqref{eq:uniform-p-variation-transformation} differs from the
corresponding full-increment sum in at most one term. Since
$I^\nabla(f,X)$ is continuous,
\[
 \max_i\sup_{u,v\in[t_i^n,t_{i+1}^n]}
 |I_u^\nabla(f,X)-I_v^\nabla(f,X)|
 \longrightarrow0.
\]
The difference between the stopped and full-increment sums is
therefore bounded uniformly in $t$ by
\[
 2
 \left(
 \max_i\sup_{u,v\in[t_i^n,t_{i+1}^n]}
 |I_u^\nabla(f,X)-I_v^\nabla(f,X)|
 \right)^p,
\]
which converges to zero. Equation
\eqref{eq:uniform-p-variation-transformation} now follows from
\eqref{eq:variation-of-fX}.
\end{proof}

\begin{remark}[$p=2$: a pathwise It\^o isometry]
{\em For $p=2$, the preceding proposition gives
\[
    d[I^\nabla(f,X)]_{\pi,t}^{(2)}
    =
    \left\langle
        (df_{X_t})^{\otimes 2},
        d[X]_{\pi,t}^{(2)}
    \right\rangle.
\]
Thus the quadratic variation of the pathwise It\^o integral is
obtained by contracting the second variation tensor of the path
with the tensor square of the integrand.

In particular, let $B$ be Brownian motion on a Riemannian manifold
$(M,g)$ and let $\nabla$ be the Levi--Civita connection. Then  along any sequence of partitions along which the quadratic variation
of $B$ is realized pathwise, we have \cite[Chap. V]{ikedawatanabe} 
\[
    d[B]_{\pi,t}^{(2)}
    =
    g^{-1}_{B_t}\,dt.
\]
Hence by Proposition \ref{prop:uniform-p-variation},
\[
\begin{aligned}
    d[I^\nabla(f,B)]_{\pi,t}^{(2)}
    &=
    \left\langle
        (df_{B_t})^{\otimes 2},
        g^{-1}_{B_t}
    \right\rangle dt
    =
    |df_{B_t}|_{g^{-1}}^2\,dt
    =
    |\operatorname{grad}f(B_t)|_g^2\,dt.
\end{aligned}
\] where ${\rm grad} f=(df)^{\#}=g^{-1}(df,.) \in TM.$ Therefore
\[
    [I^\nabla(f,B)]_\pi^{(2)}([0,t])
    =
    \int_0^t
    |\operatorname{grad}f(B_s)|_g^2\,ds.
\]
This is a pathwise version of the It\^o isometry, derived in \cite{ananova2017} in the Euclidean case: the second variation
of the pathwise It\^o integral is exactly the accumulated squared norm
of the integrand. When the pathwise integral is identified with the
classical stochastic It\^o integral, taking expectations yields, under
the usual square-integrability assumptions, the familiar It\^o isometry relation \cite{ikedawatanabe}:
\[
    \mathbb E
    \left[
        \left|
            \int_0^t df(B_s)\,d^\nabla B_s
        \right|^2
    \right]
    =
    \mathbb E
    \left[
        \int_0^t
        |\operatorname{grad}f(B_s)|_g^2\,ds
    \right]
\] }
\end{remark}

\begin{remark}[Even versus odd orders]
\label{rem:even-odd-variation}
{\rm When $p=2m$ is even, the absolute $p$-energy measure is obtained by
contracting the tensor-valued discrete variation with the positive
covariant tensor
\[  \operatorname{Sym}(g^{\otimes m}).
\]
Consequently, convergence of the tensor variation measures
automatically implies convergence of the positive energy measures:
\[
 \nu_n\rightharpoonup
 \left\langle
 \operatorname{Sym}(g^{\otimes m}),d[X]_\pi^{(p)}
 \right\rangle.
\]
Thus, in the even-order case, atomlessness of $[X]_\pi^{(p)}$ is
sufficient for uniform localization.
For even $p$, equation
\eqref{eq:p-variation-transformation} may be interpreted as an
isometry relation, analogous to the relations obtained in the Euclidean case in 
 \cite[Thm 2.1]{ananova2017} and \cite[Thm~2.1]{CP2019}:
 \begin{equation}
    [I^\nabla(f,X)]_\pi^{(p)}([0,t])
    =
    \int_{[0,t]}
    \left\langle
        (df_{X_s})^{\otimes p},
        d[X]_{\pi,s}^{(p)}
    \right\rangle.\label{eq:manifold-isometry-global}
\end{equation}
The relation \eqref{eq:manifold-isometry-global} states that the pathwise integral  \eqref{eq:pathwiseintegral} preserves f $p-$th variation:  the $p$-th variation of the integral is equal to (the $p-$th power of) a weighted $L^p$-seminorm of  $df$. 
Note that only the first-order term $df$ is involved in this relation.

For odd $p$, the tensor variation is signed. Weak convergence of
\[
             \sum_i(\xi_i^n)^{\otimes p}\delta_{t_i^n}
\]
does not control the positive measures
\[
             \sum_i|\xi_i^n|_g^p\delta_{t_i^n},
\]
because large positive and negative contributions may cancel in the
tensor limit. The additional convergence assumption on the absolute
$p$-energy measures excludes such concentration and cancellation.
The right-hand side is generally signed and does not
define a seminorm. In that case,
\eqref{eq:p-variation-transformation} is more appropriately called a
$p$-th variation identity.}
\end{remark}

\section{Higher-order tangent geometry and a transfer principle}
\label{sec:transfer-principle}

The change of variable formula obtained in the previous section admits
a geometric interpretation which extends the second-order differential
geometry introduced by Schwartz \cite{schwartz1982,schwartz1984}
; see also
Émery \cite{emery}. The purpose of this section is to formulate
this interpretation for an arbitrary integer order.
\subsection{Higher-order tangent vectors}

Denote by $J_x^p(M,\mathbb R)$ the space of
$p$-jets at $x\in M$ of smooth real-valued functions on $M$. Two smooth
functions $f,g$ define the same $p$-jet at $x$ if, in any local
coordinate system around $x$, all their partial derivatives of order
at most $p$ coincide at $x$.
We define the space of \emph{$p$-th order tangent vectors at $x$} by
\begin{equation}
    \tau_x^{(p)}M
    :=
    \left(
        J_x^p(M,\mathbb R)/\mathbb R
    \right)^*.
    \label{eq:p-order-tangent-space}
\end{equation}
Equivalently, $\tau_x^{(p)}M$ is the vector space of differential
operators
$L:C^\infty(M)\longrightarrow\mathbb R$
of order at most $p$, evaluated at $x$, which annihilate constants.
The spaces $\tau_x^{(p)}M$ form a finite-dimensional vector bundle $
    \tau^{(p)}M\longrightarrow M.$
$\tau^{(1)}M=TM$ is the  tangent bundle, $\tau^{(2)}M$ is the second-order tangent bundle appearing
in the differential geometry of Schwartz \cite{schwartz1982,schwartz1984}.
There is a natural filtration
\[
    TM=\tau^{(1)}M
    \subset\tau^{(2)}M
    \subset\cdots
    \subset\tau^{(p)}M,
\]
and the principal-symbol map yields the canonical short exact sequence
\begin{equation}
    0
    \longrightarrow
    \tau^{(p-1)}M
    \longrightarrow
    \tau^{(p)}M
    \overset{\sigma_p}{\longrightarrow}
    \Sym^p(TM)
    \longrightarrow
    0.
\end{equation}
Thus
\[
    \tau^{(p)}M/\tau^{(p-1)}M
    \simeq
    \Sym^p(TM),
\]
and the highest-order component of an order-$p$ differential operator
is a symmetric contravariant $p$-tensor.
This higher-order tangent bundle has a classical antecedent in the
osculating bundles of Pohl \cite{Pohl1962} and the jet bundle of Ehresmann \cite{ehresmann2,saunders1989}. In the terminology of \cite{Pohl1962}, the bundle of
$p$-th order tangent vectors is the dual of the bundle of $p$-jets
of functions and fits into the natural exact sequence
\[
    0\longrightarrow \tau^{(p-1)}M
    \longrightarrow \tau^{(p)}M
    \longrightarrow \Sym^p(TM)
    \longrightarrow0.
\]
As emphasized by Pohl \cite{Pohl1962},   this filtration has {\it no canonical splitting}
determined by the smooth structure alone.
\subsection{Splitting by a connection}
Let $\nabla$ be an affine connection on $M$. Recall the symmetrized
covariant derivatives
\[
    \nabla^{(k)}f:=\Sym(\nabla^k f)
    \in\Gamma(\Sym^k(T^*M)).
\]
For $f\in C^p(M,\mathbb R)$, we denote
\begin{equation}
   T_\nabla^{p-1}f
    :=
    \left(
        \nabla f,
        \frac1{2!}\nabla^{(2)}f,
        \ldots,
        \frac1{(p-1)!}\nabla^{(p-1)}f
    \right),  \label{eq:covariant-truncated-jet}
\end{equation}
for its covariant $(p-1)$-jet. Thus the manifold analogue of the Taylor jet used
in~\cite{CP2019} is
\[
    T_\nabla^{p-1}f
    \in
    \Gamma\left(
        T^*M\oplus\Sym^2(T^*M)\oplus\cdots
        \oplus\Sym^{p-1}(T^*M)
    \right).
\]
The factorial normalization is chosen so
that its pairing with tensor increments coincides with the
corresponding terms in the covariant Taylor expansion.
Dualizing and removing the zeroth-order component gives a
connection-dependent splitting
\begin{equation}
    \tau_x^{(p)}M
    \simeq
    T_xM
    \oplus
    \Sym^2(T_xM)
    \oplus\cdots\oplus
    \Sym^p(T_xM).
    \label{eq:connection-tangent-splitting}
\end{equation}
The relation between this covariant representation and the intrinsic
$p$-jet is given by the following splitting result.
\begin{proposition}[Connection splitting of higher-order tangent vectors]
\label{prop:connection-splitting}
Let $M$ be a smooth manifold equipped with an affine connection
$\nabla$, and let $p\geq1$. Then the map
\[
    J_x^p(M,\mathbb R)
    \longrightarrow
    \mathbb R
    \oplus
    \bigoplus_{k=1}^p \Sym^k(T_x^*M),
    \qquad
    j_x^p f
    \longmapsto
    \bigl(
        f(x),\nabla f(x),\ldots,\nabla^{(p)}f(x)
    \bigr),
\]
is a linear isomorphism. Consequently, dualizing and removing the
zeroth-order component yields a connection-dependent bundle
isomorphism
\[
    S_p^\nabla:
    \tau^{(p)}M
    \longrightarrow
    \bigoplus_{k=1}^p\Sym^k(TM).
\]
Thus every $L\in\tau_x^{(p)}M$ is represented uniquely by tensors
$    A_k\in\Sym^k(T_xM), 1\leq k\leq p,$
such that
\begin{equation}
    Lf
    =
    \sum_{k=1}^p
    \frac1{k!}
    \left\langle
        \nabla^{(k)}f(x),A_k
    \right\rangle.
    \label{eq:p-tangent-connection-representation}
\end{equation}
Moreover, this splitting is compatible with the filtration
\[
    0\longrightarrow\tau^{(p-1)}M
    \longrightarrow\tau^{(p)}M
    \longrightarrow\Sym^p(TM)
    \longrightarrow0.
\]
\end{proposition}
\begin{proof}
In local coordinates, $\nabla^{(k)}f(x)$ is equal to the symmetric
$k$-th partial derivative of $f$ at $x$ plus a linear combination
of derivatives of $f$ of order strictly less than $k$. Hence the
displayed map on jets is triangular, with the identity on each
diagonal block, and is therefore an isomorphism. Dualizing gives the
stated splitting. Compatibility with the filtration follows
immediately from the construction.
\end{proof}
In particular, the compensated integral introduced in
Section~\ref{sec:change-of-variable} is naturally paired with the
lower-order covariant jet:
\[
    \int_0^T
    \left\langle
        T_\nabla^{p-1}f(X_t),
        d_\pi^{p-1}X_t
    \right\rangle .
\]
The splitting \eqref{eq:connection-tangent-splitting} depends on the
connection, whereas the differential operator $L$ itself does not.

\begin{remark}{\em 
A related use of jet bundles as the intrinsic replacement for
polynomials on manifolds appears in the construction of regularity structures
on manifolds and vector bundles by Hairer and Singh
\cite{hairersingh} who note that the jet bundle
is intrinsically filtered, while a choice of geometric structure
provides a grading in terms of covariant derivatives. This is related to the connection splitting
of Proposition~\ref{prop:connection-splitting}. Pullback
of jets under a nonlinear smooth map is upper triangular rather than
graded, reflecting the same mixing of lower-order components which
underlies the transfer principle considered here. Our use of this
geometry is different: we work with the dual of the reduced jet
bundle and use the $p$-th variation tensor to define a pathwise
functional on reduced $p$-jets.}
\end{remark}
\subsection{The case $p=2$: Schwartz's second-order geometry}

When $p=2$, an element $L\in\tau_x^{(2)}M$ takes in local coordinates
the form
\begin{equation}
    L
    =
    a^i\partial_i
    +
    \frac12 a^{ij}\partial_i\partial_j,
    \qquad
    a^{ij}=a^{ji}.
    \label{eq:second-order-vector-coordinate}
\end{equation}
The coefficients $a^{ij}$ transform tensorially, whereas the
coefficients $a^i$ do not: under a nonlinear change of coordinates
they acquire terms involving the second derivatives of the coordinate
transformation. Thus the object $(a^i,a^{ij})$
cannot be interpreted as a tangent vector together with a
symmetric $2$-tensor independently of any additional geometric
structure.

A connection removes this ambiguity. With respect to $\nabla$, one
may write intrinsically
\begin{equation}
    Lf
    =
    df_x(V)
    +
    \frac12
    \left\langle
        \nabla^{(2)}f(x),A
    \right\rangle,\quad{\rm with}\quad V\in T_xM,
    \quad
    A\in\Sym^2(T_xM).
    \label{eq:second-order-connection-splitting}
\end{equation}
 
This is the connection splitting of Schwartz's second-order tangent
vector \cite{schwartz1982}, which underlies Émery's intrinsic formulation of Itô
calculus on manifolds \cite{emery}. So, for a path $X\in V_2(M,\pi)$,   the differential expression
\[
    d_\pi^\nabla X
    +
    \frac12\,d[X]_\pi^{(2)}
\]
should  behave as a second-order tangent differential relative to the connection.
\subsection{A transfer principle}
Proposition~\ref{prop:connection-splitting} suggests a geometric
interpretation of the change of variable formula, inspired by L. Schwartz's 'transfer principle' \cite{emery2006,schwartz1984}. Define the order-$p$ differential along $X$ 
\begin{equation}
\boxed{
    \mathbf{d}_\pi^{(p)}X
    :=
    d_\pi^{p-1}X\,\nabla_{p-1}
    +
    \frac1{p!}
    d[X]_\pi^{(p)}\,\nabla^{(p)} .
}
    \label{eq:p-order-differential-formal}
\end{equation}
The notation is symbolic: the first term denotes the lower-order
compensated differential acting on the covariant $(p-1)$-jet, while
the second term is the contraction of the $p$-th variation tensor with
the symmetric $p$-th covariant derivative.
Its action on $f\in C^p(M,\mathbb R)$ is defined by
\begin{equation}
\begin{split}
    \left\langle
        j^pf(X_t),
        \mathbf{d}_\pi^{(p)}X_t
    \right\rangle
    &:=
    \left\langle
        T_\nabla^{p-1}f(X_t),
        d_\pi^{p-1}X_t
    \right\rangle +
    \frac1{p!}
    \left\langle
        \nabla^{(p)}f(X_t),
        d[X]_{\pi,t}^{(p)}
    \right\rangle .
    \label{eq:p-order-differential-action}
\end{split}
\end{equation}
The change of variable formula may then be written formally as
\begin{equation}
    f(X_T)-f(X_0)
    =
    \int_0^T
    \left\langle
        j^pf(X_t),
        \mathbf{d}_\pi^{(p)}X_t
    \right\rangle.
    \label{eq:integrated-transfer-chain-rule}
\end{equation}
\paragraph{Transfer principle} {\em The combination
\begin{equation*}
     d_\pi^{p-1}X\,\nabla_{p-1}
    +
    \frac1{p!}
    d[X]_\pi^{(p)}\,\nabla^{(p)}
\end{equation*}
behaves as an intrinsic
order-$p$ tangent vector along $X$ under smooth change  of coordinates.}
\vskip 1cm
Although the individual lower-order components appearing
in a coordinate representation need not transform tensorially, their
combination with the highest-order tensor
$d[X]_\pi^{(p)}$ defines an object whose action depends only on the reduced
$p$-jet of the test function.

The statement above is formal: the mathematical content of this transfer principle is an intrinsic jet formulation of Theorem \ref{thm:change-of-variable}:
\begin{proposition}[Intrinsic jet formulation of the change-of-variable formula]
\label{prop:transfer-principle}
Let $p\geq2$, let $X\in V_p(M,\pi)$, and let $\nabla$ be an
affine connection on $M$. For $f\in C^p(M,\mathbb R)$, set
\[
\begin{aligned}
\mathcal D_{\pi,X}^{(p),\nabla}(f)
:={}&
\int_0^T
\left\langle
T_\nabla^{p-1}f(X_t),
d_\pi^{p-1}X_t
\right\rangle
+
\frac1{p!}
\int_0^T
\left\langle
\nabla^{(p)}f(X_t),
d[X]_{\pi,t}^{(p)}
\right\rangle .
\end{aligned}
\]
Then $\mathcal D_{\pi,X}^{(p),\nabla}(f)$ depends only on the
reduced $p$-jet
\[
t\longmapsto
[j_{X_t}^p f]
\in
X^*\bigl(J^p(M,\mathbb R)/\mathbb R\bigr),
\]
and is independent of the choice of affine connection $\nabla$.
Moreover,
\[
\mathcal D_{\pi,X}^{(p),\nabla}(f)
=
f(X_T)-f(X_0).
\]
The connection-dependent decomposition into the 
lower-order term and the $p$-th variation term represents an
intrinsic linear functional of the reduced $p$-jet of $f$ along $X$.
\end{proposition}
\begin{proof}
The first assertion follows from Proposition~\ref{prop:connection-splitting},
since the covariant derivatives appearing in the definition of
$\mathcal D_{\pi,X}^{(p),\nabla}$ are the connection-dependent
components of the reduced $p$-jet. By
Theorem~\ref{thm:change-of-variable},
\[
    \mathcal D_{\pi,X}^{(p),\nabla}(f)
    =f(X_T)-f(X_0),
\]
and the right-hand side is independent of $\nabla$.
\end{proof}
\paragraph{Intuition behind the transfer principle}

The order-$p$ interpretation is already visible at the level of a
single 'logarithmic' increment.
Let
\[
    \xi=e_\nabla(x,y)=\exp_x^{-1}(y).
\]
Associated with $(x,y)$ define the differential operator
\begin{equation}
    \mathcal T_{x,y}^{(p)}
    :=
    \sum_{k=1}^p
    \frac1{k!}
    \xi^{\otimes k}\nabla^{(k)}
    \in\tau_x^{(p)}M.
    \label{eq:discrete-order-p-vector}
\end{equation}
Its action on $f$ is
\[
    \mathcal T_{x,y}^{(p)}f
    =
    \sum_{k=1}^p
    \frac1{k!}
    \left\langle
        \nabla^{(k)}f(x),
        \xi^{\otimes k}
    \right\rangle.
\]
The covariant Taylor formula gives
\begin{equation}
    f(y)-f(x)
    =
    \mathcal T_{x,y}^{(p)}f
    +
    o(|\xi|^p).
    \label{eq:discrete-transfer-taylor}
\end{equation}
Thus every small displacement on the manifold determines an order-$p$ tangent vector up to an
error of order $o(|\xi|^p)$. 
This discrete order-$p$ tangent structure motivates the symbolic
notation \eqref{eq:p-order-differential-formal}; after summation, Proposition~\ref{prop:transfer-principle}
identifies the resulting action on $p$-jets.

\begin{remark}[Relation with Schwartz--Émery geometry]
{\em For $p=2$, \eqref{eq:p-order-differential-formal} becomes
\[
    d_\pi^\nabla X\,\nabla
    +
    \frac12
    d[X]_\pi^{(2)}\,\nabla^{(2)},
\]
which is exactly the structure of a Schwartz second-order tangent
vector \cite{schwartz1984}. In stochastic calculus, the quadratic variation supplies the
principal symbol of the second-order differential, while the
first-order component depends on the chosen connection. This is the
geometric mechanism underlying Émery's formulation \cite{emery} of Itô calculus on
manifolds.

The construction above shows that the same mechanism extends naturally
to arbitrary order $p$: the $p$-th variation tensor  supplies the canonical highest-order component corresponding to the principal-symbol part of the higher-order tangent geometry.}
\end{remark}
\begin{remark}[Exact sequence and connection-dependent splitting]
\label{rem:exact-sequence-splitting}
{\em 
For each $p\geq2$, the filtration of differential operators by order
gives the canonical short exact sequence
\begin{equation}
\boxed{ 0
\longrightarrow
\tau^{(p-1)}M
\overset{\iota_{p-1,p}}{\longrightarrow}
\tau^{(p)}M
\overset{\sigma_p}{\longrightarrow}
\Sym^p(TM)
\longrightarrow
0 .
}
\label{eq:exact-sequence-order-p}
\end{equation}
Here $\iota_{p-1,p}$ is the natural inclusion of differential
operators of order at most $p-1$ into differential operators of order
at most $p$, and $\sigma_p$ is the principal-symbol map. Thus canonically
\begin{equation}
    \tau^{(p)}M/\tau^{(p-1)}M
    \simeq
    \Sym^p(TM).
\label{eq:quotient-principal-symbol}
\end{equation}
A choice of affine connection $\nabla$ provides a splitting of
\eqref{eq:exact-sequence-order-p}. With the convention
of \eqref{eq:p-tangent-connection-representation}, there is a bundle
isomorphism
\[
    S_p^\nabla:
    \tau^{(p)}M
    \longrightarrow
    \bigoplus_{k=1}^{p}\Sym^k(TM),
\]
under which an element $L\in\tau_x^{(p)}M$ is represented by
\[
    S_p^\nabla(L)
    =
    (A_1,\ldots,A_p),
    \qquad
    A_k\in\Sym^k(T_xM),
\]
through
\[
    Lf
    =
    \sum_{k=1}^{p}
    \frac1{k!}
    \left\langle
        \nabla^{(k)}f(x),A_k
    \right\rangle.
\]

Under this splitting, the exact sequence
\eqref{eq:exact-sequence-order-p} becomes the evident split exact
sequence
\[
\begin{tikzcd}[column sep=large,row sep=large]
0
\arrow[r]
&
\tau^{(p-1)}M
\arrow[r,"\iota_{p-1,p}"]
\arrow[d,"S_{p-1}^\nabla"']
&
\tau^{(p)}M
\arrow[r,"\sigma_p"]
\arrow[d,"S_p^\nabla"']
&
\Sym^p(TM)
\arrow[r]
\arrow[d,equal]
&
0
\\
0
\arrow[r]
&
\displaystyle\bigoplus_{k=1}^{p-1}\Sym^k(TM)
\arrow[r,"j_{<p}"]
&
\displaystyle\bigoplus_{k=1}^{p}\Sym^k(TM)
\arrow[r,"\tfrac{1}{p!}\operatorname{pr}_p"]
&
\Sym^p(TM)
\arrow[r]
&
0 ,
\end{tikzcd}
\]
where
\[
    j_{<p}(A_1,\ldots,A_{p-1})
    =
    (A_1,\ldots,A_{p-1},0),\qquad {\rm and}\quad
    \operatorname{pr}_p(A_1,\ldots,A_p)=A_p.
\]

The upper row is intrinsic, whereas the vertical identifications
depend on the choice of connection. In particular, a connection
provides a noncanonical decomposition
\[
    \tau^{(p)}M
    \simeq
    \tau^{(p-1)}M\oplus\Sym^p(TM),
\]
and, by iteration,
\[
    \tau^{(p)}M
    \simeq
    \bigoplus_{k=1}^{p}\Sym^k(TM).
\]

Consequently, the highest-order component of an order-$p$ tangent
vector is intrinsic, while its realization as one component of the
direct sum above depends on $\nabla$. For the order-$p$ differential
associated with a path $X$, this highest-order component is represented
by its $p$-th variation tensor $d[X]_\pi^{(p)}$, up to a normalization.}
\end{remark}
The transfer principle identifies the completed change-of-variable
functional as an intrinsic object. We now make this statement
quantitative by deriving explicit transformation formulas for its
connection-dependent compensated component.
\subsection{Transformation under smooth maps and change of connection}
\label{subsec:connection-transformation}

The connection-independence statement in Proposition~\ref{prop:transfer-principle} may be
strengthened to an explicit transformation rule for the compensated
integral. For an affine connection $\nabla$ on $M$, write
\[
 I_{\pi,X}^{\nabla,p}(f)
 :=
 \int_0^T
 \left\langle
 T_\nabla^{p-1}f(X_t),d_\pi^{p-1}X_t
 \right\rangle .
\]
Thus, by Theorem~\ref{thm:change-of-variable},
\begin{equation}
 I_{\pi,X}^{\nabla,p}(f)
 =
 f(X_T)-f(X_0)
 -
 \frac{1}{p!}
 \int_0^T
 \left\langle
 \nabla^{(p)}f(X_t),d[X]_{\pi,t}^{(p)}
 \right\rangle .
 \label{eq:integral-endpoint-representation}
\end{equation}

\begin{proposition}[Transformation of the compensated integral]
\label{prop:connection-transformation}
Let $p\geq2$, let $X\in V_p(M,\pi)$, and let
$f\in C^p(M,\mathbb R)$.
\begin{enumerate}
\item
Let $\nabla$ and $\widetilde\nabla$ be two affine connections on $M$.
Then
\begin{equation}
\begin{split}
 I_{\pi,X}^{\widetilde\nabla,p}(f)
 -
 I_{\pi,X}^{\nabla,p}(f)
 &=
 \frac{1}{p!}
 \int_0^T
 \left\langle
 \nabla^{(p)}f(X_t)
 -
 \widetilde\nabla^{(p)}f(X_t),
 d[X]_{\pi,t}^{(p)}
 \right\rangle .
\end{split}
\label{eq:connection-change-integral}
\end{equation}
Equivalently, the following functional is invariant under change of connection:
\[
 D_{\pi,X}^{(p),\nabla}(f)
 :=
 I_{\pi,X}^{\nabla,p}(f)
 +
 \frac{1}{p!}
 \int_0^T
 \left\langle
 \nabla^{(p)}f(X_t),d[X]_{\pi,t}^{(p)}
 \right\rangle.
\]
\item
More generally, let $(N,\nabla^N)$ be another manifold with an
affine connection, let $\nabla^M$ be an affine connection on $M$, and
let $\Phi:M\to N$ be smooth. For $f\in C^p(N,\mathbb R)$,
\begin{equation}
\begin{split}
&
 I_{\pi,\Phi\circ X}^{\nabla^N,p}(f)
 -
 I_{\pi,X}^{\nabla^M,p}(f\circ\Phi)
\\
&\quad =
 \frac{1}{p!}
 \int_0^T
 \left\langle
 (\nabla^M)^{(p)}(f\circ\Phi)(X_t)
 -
 (D\Phi_{X_t})^{*\otimes p}
 (\nabla^N)^{(p)}f(\Phi(X_t)),
 d[X]_{\pi,t}^{(p)}
 \right\rangle .
\end{split}
\label{eq:smooth-map-transformation}
\end{equation}
Consequently,
\begin{equation}
 D_{\pi,\Phi\circ X}^{(p),\nabla^N}(f)
 =
 D_{\pi,X}^{(p),\nabla^M}(f\circ\Phi).
 \label{eq:completed-functional-naturality}
\end{equation}
\end{enumerate}
\end{proposition}

\begin{proof}
Applying Theorem~\ref{thm:change-of-variable} with the connection $\nabla$ gives
\[
 f(X_T)-f(X_0)
 =
 I_{\pi,X}^{\nabla,p}(f)
 +
 \frac{1}{p!}
 \int_0^T
 \left\langle
 \nabla^{(p)}f(X_t),d[X]_{\pi,t}^{(p)}
 \right\rangle .
\]
Applying the same theorem with $\widetilde\nabla$ gives
\[
 f(X_T)-f(X_0)
 =
 I_{\pi,X}^{\widetilde\nabla,p}(f)
 +
 \frac{1}{p!}
 \int_0^T
 \left\langle
 \widetilde\nabla^{(p)}f(X_t),
 d[X]_{\pi,t}^{(p)}
 \right\rangle .
\]
The left-hand sides are identical, while the variation tensor
$[X]_\pi^{(p)}$ is independent of the choice of connection by
Proposition~2.5. Subtracting the first identity from the second
therefore yields
\[
 I_{\pi,X}^{\widetilde\nabla,p}(f)
 -
 I_{\pi,X}^{\nabla,p}(f)
 =
 \frac{1}{p!}
 \int_0^T
 \left\langle
 \nabla^{(p)}f
 -
 \widetilde\nabla^{(p)}f,
 d[X]_\pi^{(p)}
 \right\rangle ,
\]
which proves~\eqref{eq:connection-change-integral}. Rearranging this
identity proves the connection-independence of
$D_{\pi,X}^{(p),\nabla}$.

We next prove~\eqref{eq:smooth-map-transformation}. Let
 $Y:=\Phi\circ X.$
By Proposition~\ref{prop:smooth-map-transformation} $Y\in V_p(N,\pi)$ and
\begin{equation}
 d[Y]_{\pi,t}^{(p)}
 =
 (D\Phi_{X_t})^{\otimes p}d[X]_{\pi,t}^{(p)}.
 \label{eq:variation-map-rule-used}
\end{equation}
The change-of-variable formula on $N$ gives
\[
 I_{\pi,Y}^{\nabla^N,p}(f)
 =
 f(Y_T)-f(Y_0)
 -
 \frac{1}{p!}
 \int_0^T
 \left\langle
 (\nabla^N)^{(p)}f(Y_t),d[Y]_{\pi,t}^{(p)}
 \right\rangle .
\]
Using~\eqref{eq:variation-map-rule-used} and the duality between
$(D\Phi)^{\otimes p}$ and $(D\Phi)^{*\otimes p}$, the correction term
can be rewritten as
\[
\begin{split}
&
 \int_0^T
 \left\langle
 (\nabla^N)^{(p)}f(Y_t),d[Y]_{\pi,t}^{(p)}
 \right\rangle
 =
 \int_0^T
 \left\langle
 (D\Phi_{X_t})^{*\otimes p}
 (\nabla^N)^{(p)}f(\Phi(X_t)),
 d[X]_{\pi,t}^{(p)}
 \right\rangle .
\end{split}
\]
On the other hand, applying Theorem~\ref{thm:change-of-variable} on $M$ to $f\circ\Phi$
gives
\[
\begin{split}
 I_{\pi,X}^{\nabla^M,p}(f\circ\Phi)
 &=
 f(\Phi(X_T))-f(\Phi(X_0))
\\
&\quad -
 \frac{1}{p!}
 \int_0^T
 \left\langle
 (\nabla^M)^{(p)}(f\circ\Phi)(X_t),
 d[X]_{\pi,t}^{(p)}
 \right\rangle .
\end{split}
\]
The endpoint terms in these two identities coincide. Their
subtraction gives~\eqref{eq:smooth-map-transformation}.

Finally, adding the corresponding highest-order correction to each
side of~\eqref{eq:smooth-map-transformation} yields
\eqref{eq:completed-functional-naturality}. Equivalently, both sides
of~\eqref{eq:completed-functional-naturality} are equal to
\[
 f(\Phi(X_T))-f(\Phi(X_0)).
\]
\end{proof}
In the case \(p=2\), 
formula~\eqref{eq:connection-change-integral} takes a
particularly explicit form. Let
\[
 C(U,V):=\widetilde\nabla_UV-\nabla_UV.
\]
Then
\[
 \widetilde\nabla^{(2)}f(U,V)
 =
 \nabla^{(2)}f(U,V)
 -
 df\bigl(\operatorname{Sym}C(U,V)\bigr),
\]
and hence
\[
 I_{\pi,X}^{\widetilde\nabla,2}(f)
 -
 I_{\pi,X}^{\nabla,2}(f)
 =
 \frac12
 \int_0^T
 \left\langle
 df\circ\operatorname{Sym}C,
 d[X]_{\pi,t}^{(2)}
 \right\rangle .
\]
Thus the change in the first-order It\^o integral is exactly
compensated by the change in the covariant Hessian term.
\section{Examples}\label{sec.examples}
We now present some examples of applications of the change of variable formula to irregular paths and processes on manifolds. 
\subsection{Quadratic variation of Riemannian Brownian motion}
The case $p=2$ connects our pathwise construction with It\^o's original construction \cite{ito1962} and the
second-order stochastic differential geometry of Schwartz \cite{schwartz1982,schwartz1984} and
Émery \cite{emery}.
Let $(M,g)$ be a $d$-dimensional Riemannian manifold and let $
    B=(B_t)_{t\geq0}$
be Brownian motion on $M$, with generator
$ \frac12\Delta_g.$
For $f,h\in C^\infty(M)$, the quadratic covariation satisfies
\begin{equation}
    d[f(B),h(B)]_t
    =
    \langle df,dh\rangle_{g^{-1}}(B_t)\,dt.
    \label{eq:brownian-covariation-functions}
\end{equation}
Equivalently, the intrinsic quadratic variation tensor of $B$ is
\begin{equation}
    d[B]^{(2)}_t   =
    g^{-1}_{B_t}\,dt,
\qquad {\rm where}\quad g^{-1}_x\in\Sym^2(T_xM)
    \label{eq:brownian-qv-tensor}
\end{equation}
denotes the cometric \cite{emery}.
Indeed, contracting \eqref{eq:brownian-qv-tensor} with
$df\otimes dh$ gives
\[
    (df\otimes dh)
    \left(g^{-1}\right)
    =
    \langle df,dh\rangle_{g^{-1}},
\]
which is precisely \eqref{eq:brownian-covariation-functions}.
Consequently,
\[
    \frac12
    \left\langle
        \nabla^{(2)}f,
        d[B]^{(2)}
    \right\rangle
    =
    \frac12
    \operatorname{tr}_g(\nabla^{(2)}f)(B_t)\,dt.
\]
For the Levi--Civita connection,
\[
    \operatorname{tr}_g(\nabla^{(2)}f)
    =
    \Delta_g f,
\]
and therefore the change of variable formula for $p=2$ becomes
\begin{equation}
    f(B_t)-f(B_0)
    =
    \int_0^t
    df(B_s)\,d^\nabla B_s
    +
    \frac12
    \int_0^t
    \Delta_g f(B_s)\,ds,
    \label{eq:brownian-ito-manifold}
\end{equation}
along any sequence of partitions for which the quadratic variation is
realized pathwise.
Thus the classical It\^o correction is precisely the contraction of
the second covariant derivative of $f$ with the second variation
tensor. In the terminology of
Section~\ref{sec:transfer-principle}, the principal symbol of the
second-order differential associated with Brownian motion is $\tfrac12 g^{-1}$.
\subsection{Fractional Brownian motion lifted to a manifold}
\label{sec:example-fbm-exp}
We now give a class of  examples with nontrivial higher-order variation.
We start with a vector process with non-zero $p-$th variation tensor and lift it to a manifold via the exponential map.

The following proposition, whose proof is given in Appendix \ref{sec.FBM}, gives the almost-sure $p-$th variation tensor of fractional Brownian motion for $p=1/H$, extending a result by \cite{pratelli2006}:
\begin{proposition}[Variation tensor of fractional Brownian motion]
\label{prop:fBm-p-variation}
Let $p\geq 2, 
    H=\frac1p$ and $X^H=(\beta^1,\ldots,\beta^d)$
where
$\beta^1,\ldots,\beta^d$  independent standard real-valued fractional
Brownian motions on a probability space $(\Omega,{\cal F},\mathbb{P})$ with Hurst exponent $H=\frac1p.$
Along the sequence of uniform partitions
$
    \pi_n
    =
    \left\{
        \frac{iT}{n}:0\leq i\leq n
    \right\},$
we have
\[
  \mathbb{P}\left(  X^H\in V_p(\mathbb R^d,\pi)\ \right)=1.
\]
with $p-$th variation tensor 
\[
    d[X^H]_{\pi,t}^{(p)}
    =
    \mathfrak m_p\,dt,
    \qquad
    \mathfrak m_p
    :=
    \mathbb E[Z^{\otimes p}]
    \in {\rm Sym}^p(\mathbb R^d),\qquad Z\sim N(0,I_d).
\]
\end{proposition}
Now let $(M,g)$ be a complete $d$-dimensional Riemannian manifold, fix $o\in M$, and
identify
$    V:=T_oM$
with $\mathbb{R}^d$ endowed with the scalar product $g_o$.
Let $X^H$ be the process defined in Proposition \ref{prop:fBm-p-variation}
with
$H=\frac1p,$ where $p\in\mathbb N, p\geq2.$
Thus $X^H$ is a centered Gaussian process satisfying
\[
    \mathbb E
    \big[
      \alpha(X_t^H)\beta(X_s^H)
    \big]
    =
    R_H(s,t)
    \langle\alpha,\beta\rangle_{g_o^{-1}},
    \qquad
    \alpha,\beta\in V^*,
\]
where
\[
    R_H(s,t)
    =
    \frac12
    \left(
       s^{2H}+t^{2H}-|t-s|^{2H}
    \right).
\]

Consider the uniform partitions
\[
    \pi_n
    =
    \left\{
       t_i^n=\frac{iT}{n}:0\leq i\leq n
    \right\}.
\]
Now define the $M$-valued continuous path
$    X_t
    =
    \exp_o(X_t^H).$
Since $M$ is complete, $\exp_o$ is defined on all of $T_oM$.
By Proposition \ref{prop:smooth-map-transformation},
\begin{equation}
    d[\Phi(Y)]_{\pi,t}^{(p)}
    =
    (D\Phi_{Y_t})^{\otimes p}
    d[Y]_{\pi,t}^{(p)},
    \label{eq:variation-transformation-example}
\end{equation}
we obtain, with $\Phi=\exp_o$,
\begin{equation}
\boxed{
    d[X]_{\pi,t}^{(p)}
    =
    \left(
       D\exp_o\big|_{X_t^H}
    \right)^{\otimes p}
    \mathfrak m_p\,dt .
}
    \label{eq:exp-fbm-pvariation}
\end{equation}
Thus the $p$-th variation tensor of $X$ is absolutely continuous  with Lebesgue density
\[
    \left(
       D\exp_o\big|_{X_t^H}
    \right)^{\otimes p}
    \mathfrak m_p
    \in
    \Sym^p(T_{X_t}M).
\]
Unlike the Euclidean variation tensor $\mathfrak m_p$, this density
depends on the position of the path through the differential of the
exponential map and therefore reflects the geometry of $M$.
\begin{remark}[The exponential map and the transfer principle]
\label{rem:exponential-transfer}
{\em The process $ X=\exp_o(B^H)$
provides a concrete illustration of the transfer principle of
Section~\ref{sec:transfer-principle}. Let
\[
    \Phi=\exp_o:T_oM\longrightarrow M.
\]
By Proposition~\ref{prop:transfer-principle}, the functional
$\mathcal D_{\pi,X}^{(p),\nabla}$ is independent of the choice of
connection. We may therefore write simply
$ \mathcal D_{\pi,X}^{(p)}(f)$
for its intrinsic value. For every $f\in C^p(M)$,
\[
    \mathcal D_{\pi,\Phi\circ B^H}^{(p)}(f)
    =
    \mathcal D_{\pi,B^H}^{(p)}(f\circ\Phi).
\]
Indeed, both sides are equal to
\[
    f\bigl(\Phi(B_T^H)\bigr)
    -
    f\bigl(\Phi(B_0^H)\bigr).
\]
Thus smooth maps act naturally on the intrinsic reduced-$p$-jet
functional, even though its decomposition into lower- and
highest-order terms depends on the chosen connection.
At the level of the canonical highest-order component, this
naturality is precisely the transformation rule of
Proposition~\ref{prop:smooth-map-transformation}:
\[
    d[X]_{\pi,t}^{(p)}
    =
    \left(
        D\exp_o\big|_{B_t^H}
    \right)^{\otimes p}
    d[B^H]_{\pi,t}^{(p)}.
\]
For fractional Brownian motion with $H=1/p$, Proposition~\ref{prop:fBm-p-variation}
therefore gives
\[
    d[X]_{\pi,t}^{(p)}
    =
    \left(
        D\exp_o\big|_{B_t^H}
    \right)^{\otimes p}
    \mathfrak m_p\,dt.
\]
In the connection splitting of higher-order tangent geometry, this
tensor measure is the canonical highest-order component of the
intrinsic jet functional. The lower-order compensated term is not
obtained by applying $D\exp_o$ alone: it incorporates the higher
derivatives of the exponential map through the covariant
$(p-1)$-jet.}
\end{remark}
\subsection{A fractal path with non-zero cubic variation on the sphere}
The Gaussian examples considered above have vanishing odd-order
variation tensor, by symmetry. Non-trivial odd-order variation does,
however, occur naturally for deterministic fractal paths. We give
here an explicit example based on a construction
by Schied and Zhang \cite{SchiedZhang}.

Let $\varphi:\mathbb R\to\mathbb R$ be the 1-periodic extension of
\[
    \varphi(t)
    =
    \begin{cases}
        \dfrac{3}{2}t,
        & 0\leq t\leq \dfrac13,\\[1ex]
        \dfrac{3}{4}(1-t),
        & \dfrac13\leq t\leq1,
    \end{cases}
\]
and define
\[
    w(t)
    :=
    \sum_{m=0}^\infty
    3^{-m/3}\varphi(3^m t),
    \qquad 0\leq t\leq1.
\]
Consider the sequence of ternary partitions
\[
    \pi_n
    :=
    \left\{
        \frac{k}{3^n}:0\leq k\leq3^n
    \right\}.
\]
Schied and Zhang \cite[Theorem~3.4 and Example~3.6]{SchiedZhang}
show that $w$ has non-zero signed cubic variation along $\pi$, with
\[
    \lim_{n\to\infty}
    \sum_{k<3^nt}
    \left(
        w\left(\frac{k+1}{3^n}\right)
        -
        w\left(\frac{k}{3^n}\right)
    \right)^3
    =
    \frac{27}{256}\,t,
    \qquad 0\leq t\leq1.
\]
The absolute cubic variation is finite and non-zero. Hence, in the
notation of Section~\ref{sec.pthvariation},
\[
    w\in V_3(\mathbb R,\pi),
    \qquad
    d[w]_{\pi,t}^{(3)}
    =
    \frac{27}{256}\,dt.
\]
\paragraph{A fractal path on the sphere.}
Let $S^2\subset\mathbb R^3$ be the unit sphere with its round metric,
\[
    o=e_3=(0,0,1),
    \qquad
    v=e_1=(1,0,0)\in T_oS^2,
\]
and lift the preceding path by the Riemannian exponential map:
\[
    X_t
    :=
    \exp_o(w(t)v)
    =
    \sin(w(t))\,e_1+\cos(w(t))\,e_3.
\]
We have
\[
    D\exp_o\big|_{w(t)v}(v)=U_t =
    \cos(w(t))\,e_1-\sin(w(t))\,e_3
    \in T_{X_t}S^2,
\]
The transformation rule of Proposition~\ref{prop:smooth-map-transformation} therefore yields
\[
    d[X]_{\pi,t}^{(3)}
    =
    \frac{27}{256}\,
    U_t^{\otimes3}\,dt.
\]
So $X\in V_3(S^2,\pi)$ is a path  on $S^2$ with non-zero cubic variation tensor.
Consider now the  test function
$F:S^2\longrightarrow\mathbb R$ defined by
  $F(x)=\langle v,x\rangle^3=x_1^3. $
Along $X$, we have  $F(X_t)=\sin^3(w(t)).$
Since $ r\longmapsto
    \gamma(r):=\exp_o(rv)$
is a unit-speed geodesic and $\dot\gamma(w(t))=U_t,$
we obtain
\[
    \nabla^{(3)}F(X_t)
    \bigl(U_t,U_t,U_t\bigr)
    =
    \left.
    \frac{d^3}{dr^3}\sin^3 r
    \right|_{r=w(t)}.
\]
Using
\[
    \frac{d^3}{dr^3}\sin^3 r
    =
    \frac{27\cos(3r)-3\cos r}{4},
\]
the highest-order term in the cubic change-of-variable formula is
therefore
\[
\begin{aligned}
    \frac1{3!}
    \int_0^t
    \left\langle
        \nabla^{(3)}F(X_s),
        d[X]_{\pi,s}^{(3)}
    \right\rangle
    &=
    \frac{27}{6\cdot256}
    \int_0^t
    \nabla^{(3)}F(X_s)
    (U_s,U_s,U_s)\,ds
    \\
    &=
    \frac{27}{2048}
    \int_0^t
    \bigl(
        9\cos(3w(s))-\cos(w(s))
    \bigr)\,ds.
\end{aligned}
\]
Consequently, Theorem~\ref{thm:change-of-variable} gives
\[
\boxed{
    \sin^3(w(t))
    =
    \int_0^t
    \left\langle
        T_\nabla^2F(X_s),
        d_\pi^2X_s
    \right\rangle
    +
    \frac{27}{2048}
    \int_0^t
    \bigl(
        9\cos(3w(s))-\cos(w(s))
    \bigr)\,ds
}
\]
since $w(0)=0$. 
In this example the lower-order integral  can be made completely
explicit.  Write
\[
    \gamma(r):=\exp_o(rv)
    =\sin r\,e_1+\cos r\,e_3,
    \qquad
    X_t=\gamma(w(t)),\qquad 
    U_t:=\dot\gamma(w(t)).
\]
Since $w$ is continuous, for all sufficiently fine partitions the
increments $\Delta_i w:=w(t_{i+1})-w(t_i)$ are smaller than the
injectivity radius along the great circle. Hence the geodesic joining
$X_{t_i}$ to $X_{t_{i+1}}$ is the corresponding segment of $\gamma$,
and therefore
\[
    e_\nabla(X_{t_i},X_{t_{i+1}})
    =
    (\Delta_i w)\,U_{t_i}.
\]
For
\[
    F(x)=x_1^3,
    \qquad
    h(r):=F(\gamma(r))=\sin^3 r,
\]
differentiation along the geodesic gives
\[
    dF_{X_t}(U_t)=h'(w(t)),
    \qquad
    \nabla^{(2)}F_{X_t}(U_t,U_t)=h''(w(t)),
\]
and
\[
    \nabla^{(3)}F_{X_t}(U_t,U_t,U_t)=h'''(w(t)).
\]
Consequently the compensated integral is precisely the limit of the
single sequence of sums
\[
\begin{aligned}
    C_n(t)
    &:=
    \sum_{t_i<t}
    \left[
        h'(w(t_i))\,\Delta_i w
        +
        \frac12 h''(w(t_i))(\Delta_iw)^2
    \right]
    \\
    &=
    \sum_{t_i<t}
    \left[
        dF_{X_{t_i}}
        \bigl(e_\nabla(X_{t_i},X_{t_{i+1}})\bigr)
        +
        \frac12
        \nabla^{(2)}F_{X_{t_i}}
        \bigl(
        e_\nabla(X_{t_i},X_{t_{i+1}})^{\otimes2}
        \bigr)
    \right].
\end{aligned}
\]
It is important that the two terms in this sum are compensated
\emph{before} taking the limit; they need not define convergent
integrals separately.

Indeed, Taylor's formula gives
\[
\begin{aligned}
    h(w(t_{i+1}))-h(w(t_i))
    &=
    h'(w(t_i))\,\Delta_iw
    +
    \frac12h''(w(t_i))(\Delta_iw)^2
   +
    \frac16h'''(w(t_i))(\Delta_iw)^3
    +r_i^n,
\end{aligned}
\]
where, uniformly along the partitions,
\[
    \sum_i |r_i^n|\longrightarrow0.
\]
The latter follows from the Peano remainder and the bounded cubic
energy of $w$. Summing and using telescoping therefore yields
\[
    C_n(t)
    =
    h(w(t))-h(w(0))
    -
    \frac16
    \sum_{t_i<t}
        h'''(w(t_i))(\Delta_iw)^3
    +o(1).
\]
Since
\[
    d[w]_{\pi,s}^{(3)}
    =
    \frac{27}{256}\,ds,
\]
we obtain
\[
    \int_0^t
    \left\langle
        T_\nabla^2F(X_s),d_\pi^2X_s
    \right\rangle
    =
    \sin^3(w(t))
    -
    \frac{27}{6\cdot256}
    \int_0^t h'''(w(s))\,ds .
\]
As
\[
    h^{(3)}(r)
    =
    \frac{27\cos(3r)-3\cos r}{4},
\]
this becomes
\[
\boxed{
    \int_0^t
    \left\langle
        T_\nabla^2F(X_s),d_\pi^2X_s
    \right\rangle
    =
    \sin^3(w(t))
    -
    \frac{27}{2048}
    \int_0^t
    \bigl(9\cos(3w(s))-\cos(w(s))\bigr)\,ds .
}
\]
Thus in this example the first- and second-order Taylor terms compensate each other,
while the third-order term gives the
cubic-variation correction.
\subsection{Lifting fractional
Brownian paths to spaces of constant curvature}
\label{sec:examples-fbm-constant-curvature}
We apply the previous construction to two spaces of constant
sectional curvature, the unit sphere $\mathbb S^2$ and the hyperbolic
plane $\mathbb H^2$.

Let
$p=2m, m\in\mathbb N, H=\frac1p,$ 
and let $
    B_t^H=(\beta_t^1,\beta_t^2)_{t\in[0,T]}$
where $\beta_t^1,\beta_t^2$ are independent fractional Brownian motions with Hurst
parameter $H=1/p$, considered as a process in $T_oM$ after choosing an
orthonormal identification
\[
    (T_oM,g_o)\simeq\mathbb R^2.
\]
We consider the manifold-valued lift
$X_t=\exp_o(B_t^H)$.
Along the uniform sequence of partitions
\[
    \pi_n
    =
    \left\{
      0,\frac{T}{n},\ldots,T
    \right\},
\]
the Euclidean fractional Brownian motion has $p$-th variation tensor
\[
    d[B^H]_{\pi,t}^{(2m)}
    =
    \mathfrak m_p\,dt,\qquad \mathfrak m_{2m}
    =
    (2m-1)!!
    \Sym\left((g_o^{-1})^{\otimes m}\right).
\]
By the transformation rule for $p$-th variation tensors,
\begin{equation}
    d[X]_{\pi,t}^{(2m)}
    =
    (2m-1)!!\ 
    \Sym\left(Q_t^{\otimes m}\right)\,dt,
    \label{eq:constant-curvature-pvariation}
\end{equation}
where
\begin{equation}
    Q_t
    :=
    \left(
      D\exp_o\big|_{B_t^H}
    \right)^{\otimes2}g_o^{-1}
    \in\Sym^2(T_{X_t}M).
    \label{eq:def-Q-exp}
\end{equation}
The tensor $Q_t$ has a particularly simple expression on a surface
of constant curvature. Set
\[
    r_t:=|B_t^H|_{g_o}.
\]
For $r_t>0$, let $U_t$ denote the unit tangent vector at $X_t$ to the
radial geodesic from $o$ through $X_t$, and let $V_t$ be the unit
vector obtained by parallel transporting a unit vector orthogonal to
$B_t^H$ along this geodesic. Then
\begin{equation}
    Q_t
    =
    U_t^{\otimes2}
    +
    a_K(r_t)^2 V_t^{\otimes2},
    \label{eq:Q-constant-curvature}
\end{equation}
where
\[
    a_K(r)
    =
    \begin{cases}
        \displaystyle\frac{\sin r}{r},
        & K=1,\\[1.2ex]
        \displaystyle\frac{\sinh r}{r},
        & K=-1.
    \end{cases}
\]
The value at $r=0$ is defined by continuity.
The change of variable formula takes the form
\begin{equation}
\boxed{
\begin{split}
    f(X_T)-f(X_0)
    &=
    \int_0^T
    \left\langle
       T_\nabla^{2m-1}f(X_t),
       d_\pi^{\,2m-1}X_t
    \right\rangle
   +
    \frac1{2^m m!}
    \int_0^T
    \left\langle
    \nabla^{(2m)}f(X_t),
    Q_t^{\otimes m}
\right\rangle dt.
\end{split}}
\label{eq:constant-curvature-fbm-cov}
\end{equation}

In the orthonormal frame $(U_t,V_t)$,
\begin{equation}
\begin{split}
   \left\langle
    \nabla^{(2m)}f,
    Q_t^{\otimes m}
\right\rangle 
    =
    \sum_{j=0}^m
    {m\choose j}
    a_K(r_t)^{2(m-j)}
    \nabla^{(2m)}f
    \left(
       U_t^{\otimes2j},
       V_t^{\otimes2(m-j)}
    \right).
\end{split}
\label{eq:Q-trace-expanded}
\end{equation}
This displays explicitly the effect of curvature on the highest-order
term through the differential of the exponential map.

\subsubsection{Spherical harmonic along fractional Brownian paths on $\mathbb S^2$}
\label{sec:sphere-spherical-harmonic}

Let
$M=\mathbb S^2$
with its standard Riemannian metric and Levi--Civita connection.
We identify
$    T_o\mathbb S^2\simeq\mathbb R^2$ 
isometrically where $o=(0,0,1)$ is the north pole. Let $
    p=4, H=\frac14,$
and let
\[
   X_t=\exp_o(B_t^{1/4}),\qquad{\rm with}\quad  B^{1/4}_t=(\beta_t^1,\beta_t^2)
\]
where $\beta_t^i$ are independent  one-dimensional fractional Brownian motions with $H=1/4$. Define
\[
    r_t=|B_t^{1/4}|
    =
    \sqrt{(\beta_t^1)^2+(\beta_t^2)^2}.
\]
The exponential map of the unit sphere gives
\begin{equation}
    X_t
    =
    \left(
       \frac{\sin r_t}{r_t}\beta_t^1,
       \frac{\sin r_t}{r_t}\beta_t^2,
       \cos r_t
    \right),
    \label{eq:sphere-exp-fbm}
\end{equation}
with the usual continuous interpretation at $r_t=0$.
Along the  sequence of uniform partitions $\pi_n$, the fourth variation
tensor of $X$ is
\begin{equation}
    d[X]_{\pi,t}^{(4)}
    =
    3\,\Sym(Q_t\otimes Q_t)\,dt,
    \label{eq:sphere-fourth-variation}
\end{equation}
where
\begin{equation}
    Q_t
    =
    U_t^{\otimes2}
    +
    a_t^2 V_t^{\otimes2},
    \qquad
    a_t:=\frac{\sin r_t}{r_t},
    \label{eq:sphere-Q}
\end{equation}
and $(U_t,V_t)$ is the orthonormal radial--transverse frame along
$X_t$ determined by $B_t^{1/4}$.

\begin{figure}
    \centering
    \includegraphics[width=0.4\linewidth]{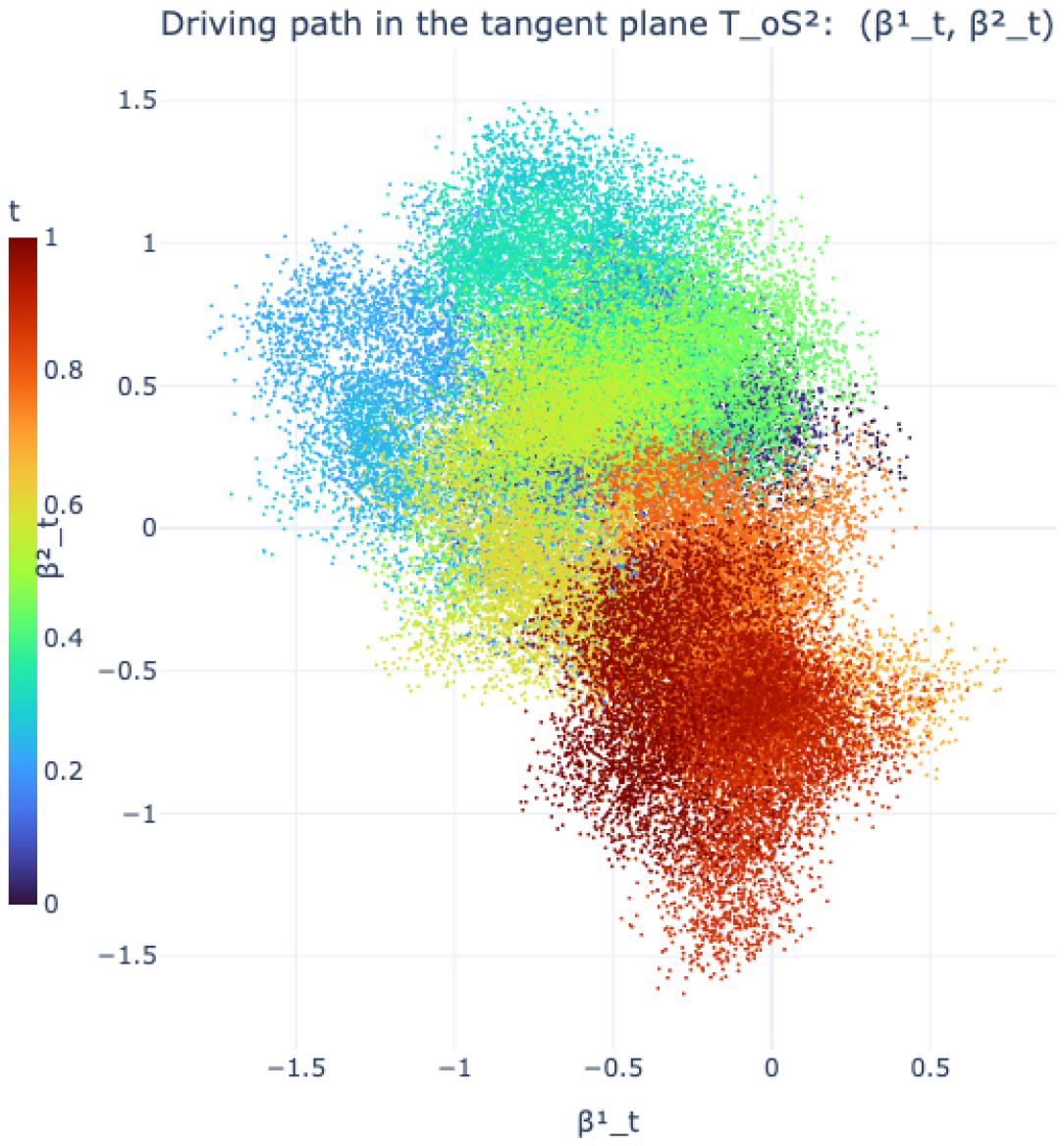}
     \includegraphics[width=0.5\linewidth]{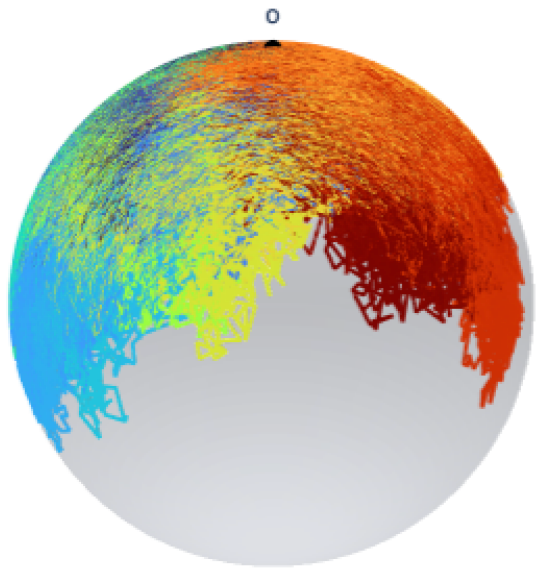}
    \caption{Left: path of $B^{\tfrac14}_t$ in the tangent plane $T_oS^2$.  Right: path of $X_t$ on the 2-sphere.}
    \label{fig:FBMS2}
\end{figure}
Let now
\[
    f(x)=x_3,
    \qquad x=(x_1,x_2,x_3)\in\mathbb S^2.
\]
This is, up to normalization, the degree-one zonal spherical harmonic
$Y_1^0$. In particular,
\[
    \Delta_{\mathbb S^2}f=-2f.
\]
From \eqref{eq:sphere-exp-fbm},
$ f(X_t)=\cos r_t.$
The covariant derivatives of $f$ take a particularly simple form.
Since $f$ is the restriction to $\mathbb S^2$ of a linear function on
$\mathbb R^3$,
\begin{equation}
    \nabla^{(2)}f=-fg.
    \label{eq:sphere-harmonic-hessian}
\end{equation}
Covariant differentiation and  symmetrization gives
\begin{equation}
    \nabla^{(3)}f
    =
    -\Sym(df\otimes g),\qquad
    \nabla^{(4)}f
    =
    f\,\Sym(g\otimes g).
    \label{eq:sphere-harmonic-fourth}
\end{equation}
Consequently the covariant third jet is
\begin{equation}
    T_\nabla^3f
    =
    \left(
       df,
       -\frac12 fg,
       -\frac16\Sym(df\otimes g)
    \right).
    \label{eq:sphere-harmonic-third-jet}
\end{equation}

The fourth-order change of variable formula therefore yields
\begin{equation}
\begin{split}
    \cos r_T-1
    &=
    \int_0^T
    \left\langle
       T_\nabla^3f(X_t),
       d_\pi^3X_t
    \right\rangle
  +
    \frac18
    \int_0^T
    \left\langle
       \nabla^{(4)}f(X_t),
       Q_t\otimes Q_t
    \right\rangle dt .
    \label{eq:sphere-harmonic-cov}
\end{split}
\end{equation}

The last term can be evaluated explicitly. With normalized
symmetrization,
\[
    \Sym(g\otimes g)(A,B,C,D)
    =
    \frac13
    \bigl(
       g(A,B)g(C,D)
       +
       g(A,C)g(B,D)
       +
       g(A,D)g(B,C)
    \bigr).
\]
It follows that, for every symmetric contravariant $2$-tensor $Q$,
\begin{equation}
    \left\langle
       \Sym(g\otimes g),
       Q\otimes Q
    \right\rangle
    =
    \frac13
    \left\{
       (\tr_g Q)^2
       +
       2\tr_g(Q^2)
    \right\}.
    \label{eq:sym-g2-Q2}
\end{equation}
Since the eigenvalues of $Q_t$ in the frame $(U_t,V_t)$ are
$1$ and $a_t^2$, we have
\[
    \tr_gQ_t=1+a_t^2,
    \qquad
    \tr_g(Q_t^2)=1+a_t^4.
\]
Thus
\begin{equation}
\begin{split}
    \left\langle
       \nabla^{(4)}f(X_t),
       Q_t\otimes Q_t
    \right\rangle
    &=
    f(X_t)
    \left(
       1+\frac23a_t^2+a_t^4
    \right)
=
    \cos r_t
    \left[
       1
       +
       \frac23
       \left(\frac{\sin r_t}{r_t}\right)^2
       +
       \left(\frac{\sin r_t}{r_t}\right)^4
    \right].\nonumber
\end{split}
\end{equation}
Hence the change of variable formula becomes
\begin{equation}
\boxed{
\begin{aligned}
    \cos r_T-1
    &=
    \int_0^T
    \left\langle
       T_\nabla^3f(X_t),
       d_\pi^3X_t
    \right\rangle
   +
    \frac18
    \int_0^T
    \cos r_t
    \left[
       1
       +
       \frac23
       \left(\frac{\sin r_t}{r_t}\right)^2
       +
       \left(\frac{\sin r_t}{r_t}\right)^4
    \right]dt .
\end{aligned}}
\label{eq:sphere-harmonic-explicit-cov}
\end{equation}
Thus, even for the degree-one spherical
harmonic, the fourth-order correction explicitly records the positive
curvature of $\mathbb S^2$.

\subsubsection{A fractional process in the hyperbolic plane $\mathbb H^2$}
Consider now the upper half-plane model
$    \mathbb H^2
    =
    \{(x_1,x_2)\in\mathbb R^2:x_2>0\}$
with metric
\[  g  =
    \frac1{x_2^2}
    \left(
      dx_1^2+dx_2^2
    \right).
\]
At $
    o=(0,1).$
 the Riemannian metric coincides with the Euclidean scalar
product, so $
    T_o\mathbb H^2\simeq\mathbb R^2$
isometrically. We have
\[
    B_t^H=(\beta_t^1,\beta_t^2),
    \qquad
    r_t
    =
    \sqrt{
       (\beta_t^1)^2+(\beta_t^2)^2
    },
\]
and, when $r_t>0$, put
\[
    c_t:=\frac{\beta_t^1}{r_t},
    \qquad
    s_t:=\frac{\beta_t^2}{r_t}.
\]
The exponential map at $o=(0,1)$ is given explicitly by
\begin{equation}
    \exp_o(r(c,s))
    =
    \left(
       \frac{c\sinh r}
            {\cosh r-s\sinh r},
       \frac{1}
            {\cosh r-s\sinh r}
    \right).
    \label{eq:hyperbolic-explicit-exp}
\end{equation}
Hence, if
\[
    X_t=(X_t^1,X_t^2)=\exp_o(B_t^H),
\]
then
\begin{equation}
    X_t^1
    =
    \frac{c_t\sinh r_t}
         {\cosh r_t-s_t\sinh r_t},
    \qquad
    X_t^2
    =
    \frac1{\cosh r_t-s_t\sinh r_t}.
    \label{eq:hyperbolic-exp-fbm-coordinates}
\end{equation}

Consider now
\[
    f(x_1,x_2)
    =
    \|(x_1,x_2)\|_2^{p}
    =
    (x_1^2+x_2^2)^m.
\]
A direct computation from
\eqref{eq:hyperbolic-exp-fbm-coordinates} gives
\begin{equation}
    (X_t^1)^2+(X_t^2)^2
    =
    \frac{
       \cosh r_t+s_t\sinh r_t
    }{
       \cosh r_t-s_t\sinh r_t
    }.
    \label{eq:euclidean-radius-hyperbolic-exp}
\end{equation}
Therefore
\begin{equation}
\boxed{
    f(X_t)
    =
    \left(
      \frac{
        \cosh r_t+s_t\sinh r_t
      }{
        \cosh r_t-s_t\sinh r_t
      }
    \right)^m .
}
\label{eq:f-hyperbolic-exp-fbm}
\end{equation}

Since $f(o)=1$, the change of variable formula
\eqref{eq:constant-curvature-fbm-cov} gives
\begin{equation}
\boxed{
\begin{split}
\left(
      \frac{
        \cosh r_T+s_T\sinh r_T
      }{
        \cosh r_T-s_T\sinh r_T
      }
\right)^m
-1
 =
\int_0^T
\left\langle
   T_\nabla^{2m-1}f(X_t),
   d_\pi^{\,2m-1}X_t
\right\rangle\\
+
\frac1{2^m m!}
\int_0^T
\left\langle
    \nabla^{(2m)}f(X_t),
    Q_t^{\otimes m}
\right\rangle\,dt 
\end{split}}
\label{eq:hyperbolic-fbm-COV}
\end{equation}
where
\begin{equation}
    Q_t
    =
    U_t^{\otimes2}
    +
    \left(
      \frac{\sinh r_t}{r_t}
    \right)^2
    V_t^{\otimes2}.
    \label{eq:hyperbolic-Q}
\end{equation}

Equivalently, the correction term in
\eqref{eq:hyperbolic-fbm-COV} is
\begin{equation}
\begin{split}
\frac1{2^m m!}
\int_0^T
\sum_{j=0}^m
{m\choose j}
\left(
  \frac{\sinh r_t}{r_t}
\right)^{2(m-j)}
&
\nabla^{(2m)}f(X_t)
\left(
 U_t^{\otimes2j},
 V_t^{\otimes2(m-j)}
\right)\,dt .
\end{split}
\label{eq:hyperbolic-explicit-correction}
\end{equation}
The factor $\frac{\sinh r_t}{r_t}$
is the transverse   distortion of the exponential map with
constant curvature $-1$. Thus the higher-order correction in the
change of variable formula detects the negative curvature of
$\mathbb H^2$ explicitly.

\paragraph{The quadratic case.}

For $p=2$, one has $H=1/2$, so $B^H$ is ordinary planar Brownian
motion and $    f(x_1,x_2)=x_1^2+x_2^2.$
In this case $d[X]_{\pi,t}^{(2)}=Q_t\,dt$
and
\begin{equation}
    f(X_T)-1
    =
    \int_0^T df(X_t)\,d_\pi^\nabla X_t
    +
    \frac12
    \int_0^T
    \left\langle
       \nabla^{(2)}f(X_t),Q_t
    \right\rangle dt .
    \label{eq:hyperbolic-p2-change}
\end{equation}

To make the last term completely explicit, set
\[
    A_t
    :=
    \cosh r_t-s_t\sinh r_t,
    \qquad
    \lambda_t
    :=
    \frac{\sinh r_t}{r_t}.
\]
Since $
    X_t^2=\frac1{A_t},$ and $ 
    \Delta_g f=4x_2^2$ 
while differentiation along the radial geodesic gives
\[
    \nabla^{(2)}f(U_t,U_t)
    =
    \frac{
       4s_t(s_t\cosh r_t-\sinh r_t)
    }{
       A_t^3
    },
\]
we obtain
\begin{equation}
\begin{split}
\frac12
\left\langle
  \nabla^{(2)}f(X_t),Q_t
\right\rangle
=
2\left\{
   \frac{\lambda_t^2}{A_t^2}
   +
   (1-\lambda_t^2)
   \frac{
      s_t(s_t\cosh r_t-\sinh r_t)
   }{
      A_t^3
   }
\right\}.
\end{split}
\label{eq:hyperbolic-p2-correction-explicit}
\end{equation}
Therefore
\begin{equation}
\frac{
   \cosh r_T+s_T\sinh r_T
}{
   \cosh r_T-s_T\sinh r_T
}
-1 =
\int_0^T df(X_t)\,d_\pi^\nabla X_t
+
2\int_0^T
\left\{
   \frac{\lambda_t^2}{A_t^2}
   +
   (1-\lambda_t^2)
   \frac{
      s_t(s_t\cosh r_t-\sinh r_t)
   }{
      A_t^3
   }
\right\}dt.
\label{eq:hyperbolic-p2-final}
\end{equation}
At $r_t=0$ the expressions above are understood by continuity.
\subsection{Fractional Brownian rotations on $\mathrm{SO}(3)$}
\label{sec:SO3}
Baudoin
and Coutin \cite{BaudouinCoutin} have proposed a construction of
Fractional Brownian motion in Lie groups for $H>1/4$. Here we study a simpler example of fractional process in a Lie group which allows to consider the entire range $0<H<1$, by applying an exponential lift of a (Euclidean) fractional Brownian motion $B^H$ in $\mathbb{R}^3$ 
 to the  Lie group $G=\mathrm{SO}(3)$.
We identify the Lie algebra
$\mathfrak g=\mathfrak{so}(3)=T_eG$  with $\mathbb R^3$ through 
\[
    v\longmapsto \widehat v,
    \qquad
    \widehat v\,w=v\times w,
\]
and equip $\mathfrak{so}(3)$ with the Ad-invariant inner product $    \langle A,B\rangle:= -\frac12\Tr(AB).$
Under the above identification this is the Euclidean scalar product
on $\mathbb R^3$, and
$   [\widehat u,\widehat v]= \widehat{u\times v}.$
We denote by $g$ the corresponding bi-invariant Riemannian metric on
$G$. Its Riemannian exponential at the identity coincides with the
Lie-group exponential.

Let $p=2m$, $m\geq1$, and let $\beta^i$ be independent fractional Brownian motion
with Hurst exponent $ H=\frac1{2m}.$ We set
\[
    B^H=(\beta^1,\beta^2,\beta^3)
\quad {\rm and}
\quad    X_t:=\exp(B_t^H).
\]
By Proposition~\ref{prop:fBm-p-variation},
\[
    d[B^H]_{\pi,t}^{(2m)}
    =
    \mathfrak m_{2m}\,dt,
    \qquad
    \mathfrak m_{2m}
    =
    (2m-1)!!
    \Sym\bigl((g_e^{-1})^{\otimes m}\bigr),
\]
and Proposition~\ref{prop:smooth-map-transformation} therefore gives
\[
    d[X]_{\pi,t}^{(2m)}
    =
    \left(
        D\exp\big|_{B_t^H}
    \right)^{\otimes 2m}
    \mathfrak m_{2m}\,dt.
\]
The differential of the exponential map may be computed explicitly.
For $v\in\mathbb R^3\simeq\mathfrak{so}(3)$, set
$ r=|v|.$
Rodrigues' formula gives
\[
    \exp(\widehat v)
    =
    I
    +
    \frac{\sin r}{r}\widehat v
    +
    \frac{1-\cos r}{r^2}\widehat v^{\,2}.
\]
Introduce the left-trivialized differential
\[
    \dexp_v
    :=
    (L_{\exp(-\widehat v)})_*
    \circ D\exp\big|_{\widehat v}
    :
    \mathfrak{so}(3)\longrightarrow\mathfrak{so}(3).
\]
It is given by
\[
    \dexp_v
    =
    \frac{1-e^{-\ad_v}}{\ad_v}
    =
    I
    -
    \frac{1-\cos r}{r^2}\widehat v
    +
    \frac{r-\sin r}{r^3}\widehat v^{\,2}.
\]
Here expressions at $r=0$ are understood by continuity.
If $v=rU$, with $|U|=1$, then
$\dexp_v U=U.$
On the orthogonal plane $U^\perp$, let
$    J_Uw:=U\times w.$
Since
\[
    \ad_v\big|_{U^\perp}=rJ_U,
    \qquad
    J_U^2=-I,
\]
we obtain, for $w\in U^\perp$,
\[
    \dexp_v w
    =
    \frac{2\sin(r/2)}{r}
    e^{-\frac r2J_U}w.
\]
Thus the differential leaves the radial direction unchanged, whereas
on the transverse plane it is the composition of a rotation through
angle $-r/2$ and the dilation
\[
    a(r):=\frac{2\sin(r/2)}{r}.
\]

Since the Gaussian tensor $\mathfrak m_{2m}$ is isotropic, the
rotational part disappears from its pushforward:
\[
    \widetilde Q_v
    :=
    (\dexp_v)^{\otimes2}g_e^{-1}=
    U^{\otimes2}
    +
    a(r)^2
    \bigl(g_e^{-1}-U^{\otimes2}\bigr).
\]
For $B_t^H\neq0$, write
\[
    r_t=|B_t^H|,
    \qquad
    U_t=\frac{B_t^H}{r_t},
\]
and define
\[
    Q_t
    :=
    (L_{X_t})_*^{\otimes2}
    \left[
        U_t^{\otimes2}
        +
        \left(
            \frac{2\sin(r_t/2)}{r_t}
        \right)^2
        \bigl(g_e^{-1}-U_t^{\otimes2}\bigr)
    \right],
\]
with $Q_t=g_{X_t}^{-1}$ when $r_t=0$. It follows that
\[
    d[X]_{\pi,t}^{(2m)}
    =
    (2m-1)!!
    \Sym\bigl(Q_t^{\otimes m}\bigr)\,dt.
\]
Consequently, for $f\in C^{2m}(G)$, the change-of-variable formula
takes the form
\[
\begin{aligned}
    f(X_T)-f(e)
    ={}&
    \int_0^T
    \left\langle
        T_\nabla^{2m-1}f(X_t),
        d_\pi^{2m-1}X_t
    \right\rangle
+
    \frac1{2^m m!}
    \int_0^T
    \left\langle
        \nabla^{(2m)}f(X_t),
        Q_t^{\otimes m}
    \right\rangle\,dt ,
\end{aligned}
\]
where $\nabla$ is the Levi--Civita connection of the bi-invariant
metric.
The factor $
    \frac{2\sin(r/2)}{r}$
measures the transverse distortion of the exponential map. In
particular, $D\exp|_v$ becomes singular in the transverse directions
when
    $|v|\in 2\pi\mathbb N.$
Thus the highest-order variation tensor records  the
degeneracy of exponential coordinates on the rotation group.
There is also a useful connection with the discussion of torsion in
Remark~\ref{rem.notorsion}. For $\lambda\in\mathbb R$, define a left-invariant affine
connection by
\[
    \nabla^\lambda_UV
    :=
    \lambda[U,V]
\]
for left-invariant vector fields $U,V$. Its torsion is
$ T^{\nabla^\lambda}(U,V)
    =
    (2\lambda-1)[U,V],$
whereas its symmetric part is the Levi--Civita connection,
\[
    \overline{\nabla^\lambda}_UV
    =
    \frac12[U,V].
\]
Hence all the connections $\nabla^\lambda$ have the same geodesics
and the same exponential map. Moreover, their symmetrized covariant
derivatives agree:
\[
    (\nabla^\lambda)^{(k)}f
    =
    (\nabla^{\mathrm{LC}})^{(k)}f,
    \qquad k\geq1.
\]
The variation tensor and the resulting change-of-variable formula are
therefore independent of $\lambda$, although the torsion varies with
$\lambda$. This gives a concrete Lie-group illustration of the fact
that torsion plays no role in the present construction.

\appendix
\section{Proof of Proposition \ref{prop:fBm-p-variation}}\label{sec.FBM}
 \begin{proof}
Let
$
    t_i^N:=\frac{iT}{N},
    \qquad 0\leq i\leq N,$
and define the normalized increments
\[
    Z_i^N
    :=
    \left(\frac{N}{T}\right)^H
    \left(
        X^H_{t_{i+1}^N}-X^H_{t_i^N}
    \right),
    \qquad 0\leq i<N.
\]
For each $N$, the random vectors $Z_i^N$ are centered Gaussian with
\begin{equation}
    \E\!\left[
        Z_i^{N,a}Z_j^{N,b}
    \right]
    =
    \delta_{ab}\rho_H(i-j),\quad
{\rm where}\quad
    \rho_H(k)
    :=
    \frac12
    \left(
        |k+1|^{2H}
        +
        |k-1|^{2H}
        -
        2|k|^{2H}
    \right).
    \label{eq:fGn-covariance}
\end{equation}
In particular, each $Z_i^N$ has the standard Gaussian law
$N(0,I_d)$.
Since $H=1/p\leq  1/2$, the sequence $\rho_H$ is absolutely summable.
For $p>2$ the second-order finite difference in
\eqref{eq:fGn-covariance} gives
$
    |\rho_H(k)|
    \leq C_H |k|^{2H-2},$
with $2H-2<-1$. Hence
\begin{equation}
    \sum_{k\in\mathbb Z}|\rho_H(k)|<\infty .
    \label{eq:fGn-summable}
\end{equation}
Let
$   Z\sim N(0,I_d),
    \quad
    \mathfrak m_p:=\E[Z^{\otimes p}]
    \in\Sym^p(\mathbb R^d).$
Fix $
    A\in\Sym^p((\mathbb R^d)^*)$
and define the polynomial
\[
    F_A(z):=A(z^{\otimes p}),
    \qquad
    \widetilde{F_A}(z)
    :=
    F_A(z)-A(\mathfrak m_p).
\]
We first establish the required convergence for a fixed $A$ and a
fixed $\varphi\in C([0,T])$.
Expand $\widetilde F_A$ in the multivariate Hermite basis:
\[
    \widetilde F_A(z)
    =
    \sum_{1\leq|\alpha|\leq p}
    c_\alpha H_\alpha(z).
\]
If $U$ and $V$ are standard Gaussian vectors in $\mathbb R^d$ with
$
    \E[U^aV^b]=r\,\delta_{ab},$
then the generating function of the Hermite polynomials gives
\[
    \E[H_\alpha(U)H_\beta(V)]
    =
    \mathbf 1_{\{\alpha=\beta\}}
    \alpha!\,r^{|\alpha|}.
\]
Consequently,
\[
    \gamma_A(k)
    :=
    \Cov\!\left(
        F_A(Z_0^N),F_A(Z_k^N)
    \right)
    =
    \sum_{1\leq|\alpha|\leq p}
    c_\alpha^2\alpha!\,
    \rho_H(k)^{|\alpha|}.
\]
Since $|\rho_H(k)|\leq1$, it follows from
\eqref{eq:fGn-summable} that
\begin{equation}
    \sum_{k\in\mathbb Z}|\gamma_A(k)|<\infty .
    \label{eq:polynomial-covariance-summable}
\end{equation}

Let
\[
    S_N(A,\varphi)
    :=
    \frac1N
    \sum_{i=0}^{N-1}
    \varphi(t_i^N)\,
    \widetilde F_A(Z_i^N).
\]
Using \eqref{eq:polynomial-covariance-summable},
\begin{align*}
    \E|S_N(A,\varphi)|^2
    &=
    \frac1{N^2}
    \sum_{i,j=0}^{N-1}
    \varphi(t_i^N)\varphi(t_j^N)
    \gamma_A(i-j)
    \leq
    \frac{\|\varphi\|_\infty^2}{N}
    \sum_{k\in\mathbb Z}|\gamma_A(k)|
    \leq
    \frac{C_{A,\varphi}}{N}.
\end{align*}
By Nelson's hypercontractivity estimate \cite[Sec. 3]{nelson1973}, if $P$ is
a polynomial of degree at most $p$ of a finite-dimensional Gaussian
vector, then
\[
    \|P\|_{L^4}\le 3^{p/2}\|P\|_{L^2},
\]
and hence
\begin{equation}
    \E |P|^4
    \le
    3^{2p}\bigl(\E |P|^2\bigr)^2.
    \label{eq:Gaussian-polynomial-fourth-moment}
\end{equation}
Since
$S_N(A,\varphi)$ is a centered polynomial of degree at most $p$ in
the Gaussian family $
    (Z_0^N,\ldots,Z_{N-1}^N),$
\eqref{eq:Gaussian-polynomial-fourth-moment} yields
\[
    \E|S_N(A,\varphi)|^4
    \leq
    \frac{C_{A,\varphi}}{N^2}.
\]
Therefore, for every $\varepsilon>0$,
\[
    \sum_{N=1}^\infty
    \mathbb{P}\bigl(
        |S_N(A,\varphi)|>\varepsilon
    \bigr)
    <\infty.
\]
By the Borel--Cantelli lemma, $S_N(A,\varphi)\to 0$ almost surely.
Since $Hp=1$,
\[
    \left(
        X^H_{t_{i+1}^N}-X^H_{t_i^N}
    \right)^{\otimes p}
    =
    \frac{T}{N}(Z_i^N)^{\otimes p}.
\]
The almost-sure convergence $S_N(A,\varphi)\to 0$ together with the 
Riemann-sum convergence, give
\begin{align}
&
    \sum_{i=0}^{N-1}
    \varphi(t_i^N)
    A\!\left(
        \left(
            X^H_{t_{i+1}^N}-X^H_{t_i^N}
        \right)^{\otimes p}
    \right)\nonumber\\
    =&
    \frac{T}{N}
    \sum_{i=0}^{N-1}
    \varphi(t_i^N)F_A(Z_i^N)
    \mathop{\longrightarrow}^{N\to\infty}
    A(\mathfrak m_p)
    \int_0^T\varphi(t)\,dt
    \qquad\text{a.s.}
    \label{eq:fBm-fixed-test}
\end{align}
We next establish finite $p$-energy. Consider the polynomial
$ G(z):=|z|^{2p}.$\\
Repeating the preceding argument with $G-\E G(Z)$ and
$\varphi\equiv1$ gives
\[
    \frac1N
    \sum_{i=0}^{N-1}|Z_i^N|^{2p}
    \mathop{\longrightarrow}^{N\to\infty}
    \E|Z|^{2p}
    \qquad\text{almost surely}.
\]
Therefore, by Cauchy--Schwarz,
\[
    \frac1N
    \sum_{i=0}^{N-1}|Z_i^N|^p
    \leq
    \left(
        \frac1N
        \sum_{i=0}^{N-1}|Z_i^N|^{2p}
    \right)^{1/2},
\]
so that, almost surely,
\[
    \sup_N
    \frac1N
    \sum_{i=0}^{N-1}|Z_i^N|^p
    <\infty.
\]
Using again $Hp=1$,
\[
    \sum_{i=0}^{N-1}
    \left|
        X^H_{t_{i+1}^N}-X^H_{t_i^N}
    \right|^p
    =
    \frac{T}{N}
    \sum_{i=0}^{N-1}|Z_i^N|^p,
\]
and hence
\begin{equation}
    \sup_N
    \sum_{i=0}^{N-1}
    \left|
        X^H_{t_{i+1}^N}-X^H_{t_i^N}
    \right|^p
    <\infty
    \qquad\text{almost surely}.
    \label{eq:fBm-finite-p-energy}
\end{equation}

It remains to choose a single probability-one event on which the
tensor-valued convergence holds for every continuous test field.
Let $
    A_1,\ldots,A_m$
be a basis of the finite-dimensional space
$\Sym^p((\mathbb R^d)^*)$, and let
$    \mathcal D\subset C([0,T])$
be a countable subset which is dense for the uniform norm.
Applying \eqref{eq:fBm-fixed-test} to the countable collection of
pairs $(A_j,\varphi),1\leq j\leq m,\quad
    \varphi\in\mathcal D,$
and intersecting the corresponding probability-one events with the
event in \eqref{eq:fBm-finite-p-energy}, we obtain a single event
$\Omega_0$ of probability one on which all these convergences and
the uniform $p$-energy bound hold simultaneously.
Fix $\omega\in\Omega_0$. For $\varphi,\psi\in C([0,T])$,
\begin{align*}
&
\left|
    \sum_{i=0}^{N-1}
    (\varphi-\psi)(t_i^N)
    A_j\!\left(
        \left(
            X^H_{t_{i+1}^N}-X^H_{t_i^N}
        \right)^{\otimes p}
    \right)
\right|
\leq
    \|A_j\|\,\|\varphi-\psi\|_\infty
    \sum_{i=0}^{N-1}
    \left|
        X^H_{t_{i+1}^N}-X^H_{t_i^N}
    \right|^p.
\end{align*}
The last factor is uniformly bounded in $N$ by
\eqref{eq:fBm-finite-p-energy}. Uniform approximation by elements of
$\mathcal D$ therefore extends
\eqref{eq:fBm-fixed-test} from $\mathcal D$ to every
$\varphi\in C([0,T])$, simultaneously for $A_1,\ldots,A_m$.

Finally, every continuous tensor field
$    A_\cdot
    \in
    C\!\left(
        [0,T];
        \Sym^p((\mathbb R^d)^*)
    \right)$
has a unique representation
\[
    A_t
    =
    \sum_{j=1}^m\varphi_j(t)A_j,
    \qquad
    \varphi_j\in C([0,T]).
\]
By linearity,
\[
    \sum_{i=0}^{N-1}
    A_{t_i^N}
    \!\left(
        \left(
            X^H_{t_{i+1}^N}-X^H_{t_i^N}
        \right)^{\otimes p}
    \right)
    \longrightarrow
    \int_0^T A_t(\mathfrak m_p)\,dt
\]
on $\Omega_0$. Thus
\[
    \sum_{i=0}^{N-1}
    \left(
        X^H_{t_{i+1}^N}-X^H_{t_i^N}
    \right)^{\otimes p}
    \delta_{t_i^N}
    \stackrel{*}{\rightharpoonup}
    \mathfrak m_p\,dt
\]
as a $\Sym^p(\mathbb R^d)$-valued measure.
Together with \eqref{eq:fBm-finite-p-energy}, this proves 
$
    X^H\in V_p(\mathbb R^d,\pi)
$
and
$
    d[X^H]_{\pi,t}^{(p)}
    =
    \mathfrak m_p\,dt$ 
almost surely.
\end{proof}

\end{document}